\documentclass[matprg]{svjour3}
\usepackage{graphicx}
\usepackage{color}
\usepackage{bm}
\usepackage{amsfonts}
\usepackage[tight,footnotesize]{subfigure}
\usepackage{amsmath, amssymb, algorithm, algorithmic}

\usepackage{multirow}

\newcommand{\beq}{\begin{equation}}
\newcommand{\eeq}{\end{equation}}
\newcommand{\beqa}{\begin{eqnarray}}
\newcommand{\eeqa}{\end{eqnarray}}
\newcommand{\beqas}{\begin{eqnarray*}}
\newcommand{\eeqas}{\end{eqnarray*}}
\newcommand{\bi}{\begin{itemize}}
\newcommand{\ei}{\end{itemize}}
\newcommand{\ba}{\begin{array}}
\newcommand{\ea}{\end{array}}

\newcommand{\nn}{\nonumber}

\newcommand{\bbr}{\mathbb{R}}
\newcommand{\bbe}{\mathbb{E}}

\def\eqnok#1{(\ref{#1})}

\def\vgap{\vspace*{.1in}}
\def\Prob{{\rm Prob}}
\newcommand{\C}[1]{{\cal {#1}}}

\title{
Mini-batch Stochastic Approximation Methods for\\
Nonconvex Stochastic Composite Optimization
\thanks{August, 2013.
This research was partially supported by NSF
grants CMMI-1000347, CMMI-1254446, DMS-1319050, DMS-1016204 and ONR grant N00014-13-1-0036.
}}
\author{Saeed Ghadimi\thanks{sghadimi@ufl.edu,
Department of Industrial and Systems Engineering,
University of Florida, Gainesville, FL 32611.}
\and Guanghui Lan\thanks{glan@ise.ufl.edu,
http://www.ise.ufl.edu/glan,
Department of Industrial and Systems Engineering,
University of Florida, Gainesville, FL 32611.}
\and Hongchao Zhang\thanks{hozhang@math.lsu.edu ,
https://www.math.lsu.edu/$\sim$hozhang,
Department of Mathematics, Louisiana State University,
Baton Rouge, LA 70803.}
}

\begin{document}
\maketitle
% ABSTRACT-------------------------
\begin{abstract}
%In this paper, we propose a randomized stochastic projected gradient (RSPG) algorithm for constrained nonlinear
%(possibly nonconvex) stochastic programming. Our algorithm is an extension of the unconstrained stochastic
%algorithm developed by Ghadimi and Lan~\cite{GhaLan12} to handle convex constraints. Proper mini-batch of
%samples are taken at each iteration depending on the total budget of stochastic samples allowed.
%Our algorithm is based on a general distance function to allow taking
%advantage of the geometry of the feasible region. Complexity of this algorithm is established in a unified setting,
%which shows nearly optimal complexity of the algorithm for convex stochastic programming.
%A post-optimization phase is also proposed to significantly reduce the variance of the solutions returned by the algorithm.
%In addition, based on the RSPG algorithm, a stochastic gradient free algorithm, which only uses the stochastic
% zeroth-order information, has been  also discussed. Some preliminary numerical results are also provided.

This paper considers a class of constrained stochastic composite optimization problems
whose objective function is given by the summation of a differentiable (possibly nonconvex)
component, together with a certain non-differentiable (but convex) component.
In order to solve these problems, we propose a randomized stochastic projected gradient (RSPG) algorithm, in which
proper mini-batch of samples are taken at each iteration depending on the total budget of stochastic samples allowed.
The RSPG algorithm also employs a general distance function to allow taking
advantage of the geometry of the feasible region. Complexity of this algorithm is established in a unified setting,
which shows nearly optimal complexity of the algorithm for convex stochastic programming.
A post-optimization phase is also proposed to significantly reduce the variance of the solutions returned by the algorithm.
In addition, based on the RSPG algorithm, a stochastic gradient free algorithm, which only uses the stochastic
zeroth-order information, has been also discussed. Some preliminary numerical results are also provided.

\end{abstract}
%KEYWORDS--------------------------

\vgap

\noindent {\bf keywords} constrained stochastic programming, mini-batch of samples, stochastic approximation, nonconvex optimization,
stochastic programming, first-order method, zeroth-order method

\thispagestyle{plain}
\markboth{S. GHADIMI, G. LAN, AND H. ZHANG}{NONCONVEX STOCHASTIC COMPOSITE OPTIMIZATION}

%%%%%%%%%%%%%%%%%%%%%%%%%%%%%%%%%%%%%%%%%%%%
\section{Introduction}
\label{sec_intro}
%%%%%%%%%%%%%%%%%%%%%%%%%%%%%%%%%%%%%%%%%%%%
In this paper, we consider the following problem
\beq \label{NLP}
\Psi^* := \min\limits_{x \in X} \{\Psi(x) := f(x) + h(x)\},
\eeq
where $X$ is a closed convex set in Euclidean space $\bbr^n$, $f:X \to \bbr$ is continuously differentiable,
but possibly nonconvex, and $h$ is a simple convex function with known structure, but possibly nonsmooth
(e.g. $h(x) = \|x\|_1$ or $h(x) \equiv 0$).
We also assume that the gradient of $f$ is $L$-Lipschitz continuous for some $L >0$, i.e.,
\beq
\|\nabla f(y) - \nabla f(x)\| \le L\|y-x\|, \quad \mbox{for any } x, y \in X,
\eeq
and $\Psi$ is bounded below over $X$, i.e. $\Psi^*$ is finite.
Although $f$ is Lipschitz continuously differentiable, we assume that only the noisy gradient of $f$ is available
via subsequent calls to a {\sl stochastic first-order oracle} (${\cal SFO}$). Specifically, at the $k$-th call, $k \ge 1$,
for the input $x_k \in X$, ${\cal SFO}$ would output a {\sl stochastic gradient} $G(x_k, \xi_k)$, where
$\xi_k$  is a random variable whose distribution is supported on $\Xi_k \subseteq \bbr^d$.
Throughout the paper, we make the following assumptions for the Borel functions $G(x_k, \xi_k)$.

{\bf A1:} For any $k \ge 1$, we have
\beqa
&\mbox{a)}& \, \, \bbe [G(x_k, \xi_k)] = \nabla f(x_k), \label{ass1.a} \\
&\mbox{b)} & \, \, \bbe \left[ \|G(x_k, \xi_k) - \nabla f(x_k)\|^2 \right] \le \sigma^2, \label{ass1.b}
\eeqa
where $\sigma >0$ is a constant.
For some examples which fit our setting, one may refer the problems in references
\cite{Andr98-1,Fu06a,Fu02-1,GhaLan12,Glasserman91,LE90-1,Mairal09,MasBaxBarFre99,RubSha93}.

Stochastic programming (SP) problems have been the subject of intense studies for more than $50$ years.
In the seminal 1951 paper, Robbins and Monro \cite{RobMon51-1} proposed a classical stochastic approximation
(SA) algorithm for solving SP problems. Although their method has ``asymptotically optimal" rate of convergence
for solving a class of strongly convex SP problems, the practical performance of their method is often
poor (e.g., \cite[Section 4.5.3]{Spall03}). Later, Polyak \cite{pol90} and Polyak and Juditsky \cite{pol92}
proposed important improvements to the classical SA algorithms, where larger stepsizes were allowed in their methods.
Recently, there have been some important
developments of SA algorithms for solving convex SP problems (i.e., $\Psi$ in \eqnok{NLP} is a convex function).
Motivated by the complexity theory in convex optimization \cite{nemyud:83},
these studies focus on the convergence properties of SA-type algorithms in a finite number of iterations.
For example, Nemirovski et al. \cite{NJLS09-1} presented a mirror descent SA approach for solving general
nonsmooth convex stochastic programming problems. They showed that the mirror descent SA exhibits an
optimal ${\cal O} ( 1 /\epsilon^2)$ iteration complexity for solving these problems with an essentially
unimprovable constant factor.
Also, Lan \cite{Lan10-3} presented a unified optimal method for smooth, nonsmooth and stochastic optimization.
This unified optimal method also leads to optimal methods for strongly convex problems \cite{GhaLan12-2a,GhaLan13-1}.
However, all of the above mentioned methods need the convexity of the problem to establish their convergence and
cannot deal with the situations where the objective function is not necessarily convex.

When problem \eqnok{NLP} is nonconvex, the research on SP algorithms so far is very limited and still far from
mature. For the deterministic case, i.e., $\sigma = 0$ in \eqnok{ass1.b}, the complexity of the gradient descent method
for solving problem \eqnok{NLP} has been studied in \cite{CarGouToi10-1,Nest04}.
Very recently, Ghadimi and Lan~\cite{GhaLan12} proposed an SA-type algorithm coupled with a randomization scheme, namely,
a randomized stochastic gradient (RSG) method, for solving the unconstrained nonconvex SP problem, i.e., problem
\eqnok{NLP} with $h \equiv 0$ and $X = \bbr^n$.
In their algorithm, a trajectory $\{x_1, \ldots, x_N\}$ is generated by a stochastic gradient descent method,
and a solution $\bar x$ is randomly selected from this trajectory according to a certain probability distribution.
They showed that the number of calls to the ${\cal SFO}$ required by this algorithm to find an $\epsilon$-solution,
i.e., a point $\bar x$ such that $\bbe[\|\nabla f(\bar x)\|_2^2] \le \epsilon$, is bounded by ${\cal O} ( \sigma^2 /\epsilon^2)$.
They also presented a variant of the RSG algorithm, namely, a two-phase randomized stochastic gradient ($2$-RSG) algorithm
to improve the large-deviation results of the RSG algorithm. Specifically, they showed that the complexity of the
$2$-RSG algorithm for computing an {\sl $(\epsilon, \Lambda)$-solution}, i.e.,  a point $\bar x$ satisfying
$\Prob \{\|\nabla f(\bar x)\|_2^2 \le \epsilon\} \ge 1-\Lambda$,
for some $\epsilon > 0$ and $\Lambda \in (0,1)$, can be bounded by
\[
{\cal O} \left\{
 \frac{\log(1/\Lambda) \sigma^2}{\epsilon}\left[\frac{1}{\epsilon}
 +  \frac{\log(1/\Lambda)}{\Lambda}\right]
\right\}.
\]
They also specialized the RSG algorithm and presented a randomized stochastic gradient free (RSGF) algorithm for
the situations where only noisy function values are available. It is shown that the expected complexity of this
RSGF algorithm is ${\cal O} ( n \sigma^2 /\epsilon^2)$.

While the RSG algorithm and its variants can handle the unconstrained nonconvex SP problems, their convergence cannot
be guaranteed for stochastic composite optimization problems in \eqnok{NLP}
where $X \neq \bbr^n$ and/or $h(\cdot)$ is non-differentiable.
Our contributions in this paper mainly consist of developing variants
of the RSG algorithm by taking a mini-batch of samples at each iteration of our algorithm to deal with the constrained
composite problems while preserving the complexity results. More specifically, we first modify the scheme of the RSG algorithm
to propose a randomized stochastic projected gradient (RSPG) algorithm to solve constrained nonconvex stochastic composite problems.
Unlike the RSG algorithm, at each iteration of the RSPG algorithm, we take multiple samples such that the total number
of calls to the ${\cal SFO}$ to find a solution $\bar x \in X$ such that $\bbe[\|g_{_X}(\bar x)\|^2] \le \epsilon$,
is still ${\cal O} ( \sigma^2 /\epsilon^2)$, where $g_{_X}(\bar x)$
is a generalized projected gradient of $\Psi$ at $\bar x$ over X. In addition, our RSPG algorithm is in a more general
setting depending on a general distance function rather than Euclidean distance \cite{GhaLan12}.
This would be particularly useful for special structured constrained set (e.g., $X$ being a standard simplex).
Secondly, we present a two-phase randomized stochastic projected gradient ($2$-RSPG) algorithm, the RSPG algorithm with
a post-optimization phase, to improve the large-deviation results of the RSPG algorithm. And we show that
the complexity of this approach can be further improved under a light-tail assumption about the ${\cal SFO}$.
Thirdly, under the assumption that the gradient of $f$ is also bounded on $X$, we specialize the RSPG algorithm to give
a randomized stochastic projected gradient free (RSPGF) algorithm, which only uses the stochastic zeroth-order information.
Finally, we present some numerical results to show the effectiveness of the aforementioned randomized stochastic
projected gradient algorithms, including the RSPG, $2$-RSPG and RSPGF algorithms. Some practical improvements of these
algorithms have been also discussed.

The remaining part of this paper is organized as follows. We first describe some properties of the projection based on
a general distance function in Section 2. In section 3, a deterministic first-order method for problem \eqnok{NLP}
is proposed, which mainly provides a basis for our stochastic algorithms developed in later sections.
Then, by incorporating a randomized scheme, we present the RSPG and $2$-RSPG algorithms for solving the SP problem
\eqnok{NLP} in Section 4. In section 5, we discuss how to generalize the RSPG algorithm to the case
when only zeroth-order information is available. Some numerical results and discussions from implementing
our algorithms are presented in Section 6.  Finally, in Section 7, we give some concluding remarks.

{\bf Notation.}
We use $\|\cdot\|$ to denote a general norm
with associated inner product $\langle \cdot, \cdot \rangle$.
For any $p \ge 1$, $\|\cdot\|_p$ denote the standard $p$-norm in  $\bbr^n$, i.e.
\[
\|x\|_p^p =  \sum_{i=1}^n |x_i|^p, \qquad \mbox{for any } x \in \bbr^n.
\]
For any convex function $h$, $\partial h(x)$ is the subdifferential set at $x$.
Given any $\Omega \subseteq \bbr^n$, we say $f \in {\cal C}_L^{1,1}(\Omega)$, if $f$
is Lipschitz continuously differentiable with Lipschitz constant $L>0$, i.e.,
\beq
\|\nabla f(y) - \nabla f(x)\| \le L\|y-x\|, \qquad \mbox{for any } x, y \in \Omega,
\eeq
which clearly implies
\beq \label{smooth}
|f(y) - f(x) - \langle \nabla f(x), y - x \rangle | \le \frac{L}{2} \|y - x\|^2,
\qquad \mbox{for any } x, y \in \Omega.
\eeq
For any real number $r$, $\lceil r \rceil$ and $\lfloor r \rfloor$ denote the nearest integer to
$r$ from above and below, respectively. $\bbr_+$ denotes the set of nonnegative real numbers.
%
%
%%%%%%%%%%%%%%%%%%%%%%%%%%%%%%%%%%%%%%%%%%%%
\section{Some properties of generalized projection}
\label{sec_prox}
%%%%%%%%%%%%%%%%%%%%%%%%%%%%%%%%%%%%%%%%%%%%
In this section, we review the concept of projection in a general sense as well as its important properties.
This section consists of two subsections.
We first discuss the concept of prox-function and its associated projection in Subsection~\ref{sec_prelim}.
Then, in Subsection~\ref{sec_comp_proj}, we present some important properties of the projection, which will
play a critical role for the proofs in our later sections.

\subsection{Prox-function and projection} \label{sec_prelim}
It is well-known that using a generalized distance generating function, instead of the usual Euclidean distance function,
would lead to algorithms that can be adjusted to the geometry of the feasible set
and/or efficient solutions of the projection \cite{AuTe06-1,BBC03-1,Breg67,Lan10-3,NJLS09-1,Teb97-1}.
Hence, in this paper we would like to set up the projection based on the so-called prox-function.

A function $\omega:\,X\to \bbr$ is said to be a {\em distance generating function} with modulus $\alpha>0$
with respect to $\|\cdot\|$, if $\omega$ is continuously differentiable and strongly convex satisfying
\beq\label{strg_cnvx}
 \langle x-z, \nabla \omega(x)-\nabla\omega(z) \rangle \ge \alpha
\|x-z\|^2,\;\;\forall x,z\in X.
\eeq
Then, the {\em prox-function} associated with $\omega$ is defined as
\beq\label{prox_fun}
V(x,z)=\omega(x)-[\omega(z)+\langle \nabla \omega(z), x-z \rangle].
\eeq
In this paper, we assume that the prox-function $V$ is chosen such that the generalized projection problem given by
\beq \label{comp_proj}
x^+ = \arg\min\limits_{u \in X} \left\{ \langle g, u \rangle
+ \frac{1}{\gamma} V(u, x)+h(u) \right\}
\eeq
is easily solvable for any $\gamma>0$, $g \in \bbr^n$ and $x \in X$.
Apparently, different choices of $\omega$ can be used in the definition of prox-function.
One simple example would be $\omega(x) = \|x\|_2^2/2$,  which gives  $V(x,z) = \|x -z\|_2^2 / 2$.
And in this case, $x^+$ is just the usual Euclidean projection.
Some less trivial examples can be found, e.g., in \cite{AuTe06-1,BenMarNem01,DangLan12-1,JudNem11,nemyud:83}.
%\begin{example} Let $X = \{x \in \bbr^n: \|x\|_1 \le 1\}$. If $ \omega(x) = \frac{1}{2}\|x\|_p^2$
%with $p = 1 + 1 / \ln n$, then $\omega$ is strongly convex with modulus $\alpha = 1/{(e^2 \ln n)}$
%with respect to $\| \cdot \|_1$.
%\end{example}

\subsection{Properties of Projection} \label{sec_comp_proj}
In this subsection, we discuss some important properties of the generalized projection defined in \eqnok{comp_proj}.
Let us first define
\beq \label{proj_g}
P_X(x, g, \gamma) = \frac{1}{\gamma}(x-x^+),
\eeq
where $x^+$ is given in \eqnok{comp_proj}. We can see that $P_X(x, \nabla f(x), \gamma)$ can be viewed as a
generalized projected gradient of $\Psi$ at $x$. Indeed, if $X = \bbr^n$ and $h$ vanishes, we would
have $P_X(x, \nabla f(x), \gamma) = \nabla f(x) = \nabla \Psi(x)$.

The following lemma provides a bound for the size of $P_X(x, g, \gamma)$.
\begin{lemma} \label{proj_g_size}
Let $x^+$ be given in \eqnok{comp_proj}. Then, for any $x \in X$, $g \in \bbr^n$ and $\gamma>0$, we have
\beq \label{comp_proj_ineq}
\langle g , P_X(x, g, \gamma) \rangle \ge \alpha \|P_X(x, g, \gamma)\|^2+\frac{1}{\gamma} \left[h(x^+) -h(x) \right].
\eeq
\end{lemma}
\begin{proof}
By the optimality condition of \eqnok{comp_proj} and the definition of prox-function in \eqnok{prox_fun},
there exists a $p \in \partial h(x^+)$ such that
\[
\langle g + \frac{1}{\gamma} \left[\nabla \omega(x^+) - \nabla \omega(x) \right]+p, u-x^+ \rangle \ge 0, \qquad
 \mbox{for any } u \in X.
\]
Letting $u=x$ in the above inequality, by the convexity of $h$ and \eqnok{strg_cnvx}, we obtain
\beqa
\langle g , x-x^+ \rangle &\ge& \frac{1}{\gamma} \langle \nabla \omega(x^+) - \nabla \omega(x) , x^+ - x\rangle +
\langle p ,x^+ - x \rangle \nonumber \\
&\ge& \frac{\alpha}{\gamma} \|x^+-x\|^2+  \left[h(x^+)-h(x) \right], \nonumber
\eeqa
which in the view of \eqnok{proj_g} and $\gamma>0$ clearly imply \eqnok{comp_proj_ineq}.
\end{proof}

\vgap

It is well-known \cite{RocWet98} that the Euclidean projection is Lipschitz continuous.
Below, we show that this property also holds for the general projection.
\begin{lemma}
Let $x_1^+$ and $x_2^+$ be given in \eqnok{comp_proj} with $g$ replaced by $g_1$ and $g_2$ respectively. Then,
\beq \label{lip_comp_proj}
\|x_2^+ - x_1^+\| \le \frac{\gamma}{\alpha} \|g_2 -g_1\|,
\eeq
where $\alpha >0$ is the modulus of strong convexity of $\omega$ defined in \eqnok{strg_cnvx}.
\end{lemma}
\begin{proof}
By the optimality condition of \eqnok{comp_proj}, for any $u \in X$, there exist $p_1 \in \partial h(x_1^+)$ and
$p_2 \in \partial h(x_2^+)$  such that
\beqa
\langle g_1 + \frac{1}{\gamma} \left[\nabla \omega(x_1^+) - \nabla \omega(x) \right]+p_1 , u-x_1^+ \rangle \ge 0,
\label{Lip_comp_proj1}
\eeqa
and
\beqa
\langle g_2 + \frac{1}{\gamma} \left[\nabla \omega(x_2^+) - \nabla \omega(x) \right]+p_2 , u-x_2^+
\rangle \ge 0. \label{Lip_comp_proj2}
\eeqa
Letting $u=x_2^+$ in \eqnok{Lip_comp_proj1}, by the convexity of $h$, we have
\beqa \label{Lip_comp_proj3}
\langle g_1, x_2^+-x_1^+ \rangle &\ge& \frac{1}{\gamma} \langle \nabla \omega(x) - \nabla \omega(x_1^+) , x_2^+-x_1^+
\rangle +\langle p_1, x_1^+-x_2^+ \rangle  \nn \\
&\ge& \frac{1}{\gamma} \langle \nabla \omega(x_2^+) - \nabla \omega (x_1^+) , x_2^+-x_1^+ \rangle + \frac{1}{\gamma}
\langle \nabla \omega(x) - \nabla \omega(x_2^+) , x_2^+-x_1^+ \rangle \nn \\
& & + h(x_1^+) - h(x_2^+).
\eeqa
Similarly, letting $u=x_1^+$ in \eqnok{Lip_comp_proj2}, we have
\beqa \label{Lip_comp_proj4}
\langle g_2, x_1^+-x_2^+ \rangle &\ge&
\frac{1}{\gamma} \langle \nabla \omega(x) - \nabla \omega (x_2^+) , x_1^+-x_2^+ \rangle+\langle p_2, x_2^+-x_1^+ \rangle \nn \\
&\ge& \frac{1}{\gamma} \langle \nabla \omega(x) - \nabla \omega (x_2^+) , x_1^+-x_2^+ \rangle+h(x_2^+)-h(x_1^+).
\eeqa
Summing up \eqnok{Lip_comp_proj3} and \eqnok{Lip_comp_proj4}, by the strong convexity \eqnok{strg_cnvx} of $\omega$,
we obtain
\[
\|g_1 - g_2\| \|x_2^+-x_1^+\| \ge \langle g_1 - g_2, x_2^+-x_1^+ \rangle \ge \frac{\alpha}{\gamma} \|x_2^+-x_1^+\|^2, \]
which gives \eqnok{lip_comp_proj}.
\end{proof}

\vgap

As a consequence of the above lemma, we have $P_X(x,\cdot,\gamma)$ is Lipschitz continuous.
\begin{proposition}\label{lip_proj_grad}
Let $P_X(x,g,\gamma)$ be defined in \eqnok{proj_g}. Then, for any $g_1$ and $g_2$ in $\bbr^n$, we have
\beq
\|P_X(x, g_1, \gamma) - P_X(x,g_2, \gamma)\| \le \frac {1}{\alpha} \|g_1-g_2\|,
\eeq
where $\alpha$ is the modulus of strong convexity of $\omega$ defined in \eqnok{strg_cnvx}.
\end{proposition}
\begin{proof}
Noticing \eqnok{proj_g}, \eqnok{Lip_comp_proj1} and  \eqnok{Lip_comp_proj2}, we have
\[
\|P_X(x, g_1, \gamma) - P_X(x,g_2, \gamma)\| = \|\frac{1}{\gamma}(x-x_1^+)- \frac{1}{\gamma}(x- x_2^+)\| = \frac{1}{\gamma} \|x_2^+ - x_1^+\| \le \frac{1}{\alpha} \|g_1 -g_2\|,
\]
where the last inequality follows from \eqnok{lip_comp_proj}.
\end{proof}

\vgap

The following lemma (see, e.g., Lemma~1 of \cite{Lan10-3} and Lemma 2 of \cite{GhaLan12-2a})
characterizes the solution of the generalized projection.
\begin{lemma}\label{comp_proj_sol}
Let $x^+$ be given in \eqnok{comp_proj}. Then, for any $u \in X$, we have
\beq
\langle g, x^+ \rangle + h(x^+)+\frac{1}{\gamma} V(x^+, x)
\le \langle g, u \rangle + h(u)+\frac{1}{\gamma} [V(u, x) - V(u, x^+)].
\eeq
\end{lemma}
%
%
%%%%%%%%%%%%%%%%%%%%%%%%%%%%%%%%%%%%%%%%%%%%
\section{Deterministic first-order methods}
\label{sec_first}
%%%%%%%%%%%%%%%%%%%%%%%%%%%%%%%%%%%%%%%%%%%%
%
In this section, we consider the problem \eqnok{NLP} with $f \in {\cal C}_L^{1,1}(X)$, and
for each input $x_k \in X$, we assume that the exact gradient $\nabla f(x_k)$ is
available. Using the exact gradient information, we give a deterministic projected gradient (PG) algorithm for
solving \eqnok{NLP}, which mainly provides a basis for us to develop the stochastic first-order algorithms
in the next section.
\vskip 0.1cm

\noindent {\bf A  projected gradient (PG) algorithm}
\begin{itemize}
\item [] {\bf Input:} Given initial point $x_1 \in X$,
total number of iterations $N$, and the stepsizes $\{\gamma_k\}$ with $\gamma_k >0$, $k \ge 1$.
\item [] {\bf Step } $k=1, \ldots, N$. Compute
\beq \label{update_CPG}
x_{k+1} = \arg \min_{u \in X} \left\{\langle \nabla f(x_k), u \rangle
+\frac{1}{\gamma_k} V(u, x_k)+h(u) \right\}.
\eeq
\item [] {\bf Output:} $x_R \in \{x_k, \ldots, x_N\}$ such that
\beq \label{BDG}
R = \arg\min_{k \in \{1,\ldots, N\}} \|g_{_{X,k}}\|,
\eeq
where the $g_{_{X,k}}$ is given by
\beq \label{proj_grad}
g_{_{X,k}} = P_X(x_k, \nabla f(x_k), \gamma_k).
\eeq
\end{itemize}

We can see that the above algorithm outputs the iterate with the minimum norm of the
generalized projected gradient. In the above algorithm, we have not specified the selection of the stepsizes $\{\gamma_k\}$.
We will return to this issue after establishing the following convergence results.

\begin{theorem} \label{main_theorem_det}
Suppose that the stepsizes $\{\gamma_k\}$ in the PG algorithm are chosen such that $ 0 < \gamma_k \le 2\alpha/ L$
with $\gamma_k < 2 \alpha/L$ for at least one $k$. Then, we have
\beq \label{main_cnvg_det}
\|g_{_{X,R}}\|^2
\le \frac{L D_{\Psi}^2}{\sum_{k=1}^N (\alpha \gamma_k - L\gamma_k^2/2)},
\eeq
where
\beq \label{def_Df}
g_{_{X,R}} = P_X(x_R, \nabla f(x_R), \gamma_R) \quad \mbox{and} \quad
D_{\Psi} := \left[\frac{\left(\Psi(x_1) - \Psi^*\right)}{L}\right]^\frac{1}{2}.
\eeq
\end{theorem}
\begin{proof}
Since $f \in {\cal C}_L^{1,1}(X)$, it follows from \eqnok{smooth}, \eqnok{proj_g},
\eqnok{update_CPG} and \eqnok{proj_grad} that for any $k = 1, \ldots, N$, we have
\beqa
f(x_{k+1}) &\le& f(x_k) + \langle \nabla f(x_k), x_{k+1}-x_k \rangle + \frac{L}{2} \|x_{k+1}-x_k\|^2 \nn \\
&=& f(x_k) - \gamma_k \langle \nabla f(x_k), g_{_{X,k}} \rangle + \frac{L}{2} \gamma_k^2 \|g_{_{X,k}}\|^2.
\eeqa
Then, by Lemma~\ref{proj_g_size} with $x=x_k$, $\gamma = \gamma_k$ and $g =\nabla f(x_k)$, we obtain
\[
f(x_{k+1}) \le f(x_k) - \left[\alpha \gamma_k \|g_{_{X,k}}\|^2+ h(x_{k+1})-h(x_k)\right]+
\frac{L}{2} \gamma_k^2 \|g_{_{X,k}}\|^2,
\]
which implies
\beq
\Psi(x_{k+1}) \le \Psi(x_k) - \left(\alpha \gamma_k - \frac{L}{2} \gamma_k^2\right) \|g_{_{X,k}}\|^2 .
\eeq
Summing up the above inequalities for $k=1, \ldots, N$, by \eqnok{BDG} and $\gamma_k \le 2 \alpha/L$, we have
\beqa
\|g_{_{X,R}}\|^2 \sum_{k=1}^N \left(\alpha \gamma_k - \frac{L}{2} \gamma_k^2\right)
 &\le& \sum_{k=1}^N \left(\alpha \gamma_k - \frac{L}{2} \gamma_k^2\right) \|g_{_{X,k}}\|^2 \nn \\
&\le& \Psi(x_1) - \Psi(x_{k+1}) \le \Psi(x_1) - \Psi^*.
\eeqa
By our assumption, $ \sum_{k=1}^N \left(\alpha \gamma_k - L \gamma_k^2/2 \right) >0$.
Hence, dividing both sides of the above inequality by $\sum_{k=1}^N \left(\alpha \gamma_k - L \gamma_k^2/2 \right)$,
we obtain \eqnok{main_cnvg_det}.
\end{proof}

\vgap

%The above algorithm provides a way to choose a constant stepsize $\gamma \in (0, 2 \alpha/L)$.
The following corollary shows a specialized complexity result
for the PG algorithm with one proper constant stepsize policy.
\begin{corollary}
Suppose that in the PG algorithm the stepsizes $\gamma_k = \alpha/L$ for all $k=1, \ldots, N$.
Then, we have
\beq \label{proj_grad_cnvg}
\|g_{_{X,R}}\|^2 \le
\frac{2 L^2 D_{\Psi}^2}{\alpha^2 N}.
\eeq
\end{corollary}
\begin{proof} With the constant stepsizes $\gamma_k = \alpha/L$ for all $k=1, \ldots, N$, we have
\beq
\frac{L D_{\Psi}^2}{\sum_{k=1}^N (\alpha \gamma_k - L\gamma_k^2/2)}
= \frac{2 L^2 D_{\Psi}^2}{N \alpha^2},
\eeq
which together with \eqnok{main_cnvg_det}, clearly imply \eqnok{proj_grad_cnvg}.
\end{proof}
%
%
%%%%%%%%%%%%%%%%%%%%%%%%%%%%%%%%%%%%%%%%%%%%
\section{Stochastic first-order methods}
\label{sec_stch_first}
%%%%%%%%%%%%%%%%%%%%%%%%%%%%%%%%%%%%%%%%%%%%
%
In this section, we consider problem \eqnok{NLP} with $f \in {\cal C}_L^{1,1}(X)$, but its exact
gradient is not available. We assume that only noisy first-order information of $f$ is available via subsequent calls
to the stochastic first-order oracle ${\cal SFO}$. In particular, given the $k$-th iteration $x_k \in X$ of our algorithm,
the ${\cal SFO}$ will output the stochastic gradient $G(x_k,\xi_k)$, where $\xi_k$ is a random vector whose distribution
is supported on  $\Xi_k \subseteq \bbr^d$. We assume the stochastic gradient $G(x_k,\xi_k)$ satisfies Assumption A1.

This section also consists of two subsections. In Subsection~\ref{sec_RSPG}, we present a stochastic variant of
the PG algorithm in Section~\ref{sec_first} incorporated with a randomized stopping criterion, called the RSPG algorithm.
Then, in Subsection~\ref{sec_2RSPG}, we describe a two phase RSPG algorithm, called the $2$-RSPG algorithm, which
can significantly reduce the large-deviations resulted from the RSPG algorithm.

\subsection{A randomized stochastic projected gradient method}\label{sec_RSPG}
Convexity of the objective function often plays an important role on establishing the convergence results for the
current SA algorithms \cite{GhaLan12-2a,GhaLan13-1,NJLS09-1,lns11,Lan10-3}.
In this subsection, we give an SA-type algorithm which does not require the convexity of the objective function.
Moreover, this weaker requirement enables the algorithm to deal with the case in which the random noises $\{\xi_k\}, k \ge1$
could depend on the iterates $\{x_k\}$.
\vskip 0.1cm

\noindent {\bf A randomized stochastic projected gradient (RSPG) algorithm}
\begin{itemize}
\item [] {\bf Input:} Given initial point $x_1 \in X$,
iteration limit $N$, the stepsizes $\{\gamma_k\}$ with $\gamma_k >0$, $k \ge 1$,
the batch sizes $\{m_k\}$ with $m_k > 0$, $k \ge 1$,
and the probability mass function $P_R$ supported on  $\{1,\ldots, N\}$.
\item [] {\bf Step } $0$. Let $R$ be a random variable with probability mass function $P_{R}$.
\item [] {\bf Step } $k=1, \ldots, R-1$. Call the ${\cal SFO}$ $m_k$ times
to obtain $G(x_k, \xi_{k,i})$, $i = 1, \ldots, m_k$, set
\beq \label{def_Gk}
G_k = \frac{1}{m_k} \sum_{i=1}^{m_k} G(x_k, \xi_{k,i}),
\eeq
and compute
\beq \label{update_RSPG}
x_{k+1} = \arg \min_{u \in X} \left\{\langle G_k, u \rangle
+\frac{1}{\gamma_k} V(u, x_k)+h(u) \right\}.
\eeq
\item [] {\bf Output:} $x_R$.
\end{itemize}

\vgap

Unlike many SA algorithms, in the RSPG algorithm we use a randomized iteration count to terminate the algorithm.
In the RSPG algorithm, we also need to specify the stepsizes $\{\gamma_k\}$,
the batch sizes $\{m_k\}$ and probability mass function $P_{R}$.
We will again address these issues after presenting some convergence results of the RSPG algorithm.
\begin{theorem} \label{main_theorem_stch}
Suppose that the stepsizes $\{\gamma_k\}$ in the RSPG algorithm are chosen such that $ 0 < \gamma_k \le \alpha/ L$
with $\gamma_k < \alpha/L$ for at least one $k$, and the probability mass function $P_R$ are chosen such that
for any $k = 1, \ldots, N$,
\beq \label{prob_fun}
P_R(k) := \Prob\{R=k\} = \frac{\alpha \gamma_k- L\gamma_k^2}
{{\sum_{k=1}^N (\alpha \gamma_k- L \gamma_k^2)}}.
\eeq
Then, under Assumption A1,
\begin{itemize}
\item [(a)] for any $N \ge 1$, we have
\beq \label{main_cnvg_stch}
\bbe[\|\tilde{g}_{_{X,R}}\|^2]
\le \frac{L D_{\Psi}^2 + (\sigma^2/\alpha) {\sum_{k=1}^N (\gamma_k/m_k)}}{{\sum_{k=1}^N (\alpha \gamma_k - L\gamma_k^2)}},
\eeq
where the expectation is taken with respect to $R$ and $\xi_{[N]} := (\xi_1,\ldots,\xi_N)$, $D_{\Psi}$ is defined
in \eqnok{def_Df}, and the stochastic projected gradient
\beq \label{proj_stch_grad}
\tilde{g}_{_{X,k}} := P_X(x_k, G_k, \gamma_k),
\eeq
with $P_X$ defined in\eqnok{proj_g};
\item [(b)] if, in addition, $f$ in problem \eqnok{NLP} is convex with an optimal solution $x^*$, and the stepsizes
$\{\gamma_k\}$ are non-decreasing, i.e.,
\beq \label{incr_stepsize}
0 \le \gamma_1 \le \gamma_2 \le ... \le \gamma_N \le \frac{\alpha}{L},
\eeq
we have
\beq \label{main_cnvg_stch_cvx1}
 \bbe \left[\Psi(x_R) - \Psi(x^*)\right]
\le \frac{(\alpha- L \gamma_1)V(x^*,x_1)+(\sigma^2/2) \sum_{k=1}^N (\gamma_k^2/m_k)}
{ \sum_{k=1}^N (\alpha \gamma_k - L\gamma_k^2)},
\eeq
where the expectation is taken with respect to $R$ and $\xi_{[N]}$.
Similarly, if the stepsizes $\{\gamma_k\}$ are non-increasing, i.e.,
\beq \label{dec_stepsize}
\frac{\alpha}{L} \ge \gamma_1 \ge \gamma_2 \ge ... \ge \gamma_N \ge 0,
\eeq
 we have
\beq \label{main_cnvg_stch_cvx2}
 \bbe \left[\Psi(x_R) - \Psi(x^*)\right]
\le \frac{(\alpha- L \gamma_N)\bar{V}(x^*)+(\sigma^2/2) \sum_{k=1}^N (\gamma_k^2 /m_k)}
{\sum_{k=1}^N (\alpha \gamma_k - L\gamma_k^2)},
\eeq
 where $\bar V(x^*):= \max_{u \in X} V(x^*,u)$.
\end{itemize}
\end{theorem}

\begin{proof}
Let $\delta_k \equiv G_k - \nabla f(x_k)$, $k \ge 1$. Since $f \in {\cal C}_L^{1,1}(X)$, it follows from
\eqnok{smooth}, \eqnok{proj_g}, \eqnok{update_RSPG} and \eqnok{proj_stch_grad} that, for any $k = 1, \ldots, N$, we have
\beqa
f(x_{k+1}) &\le& f(x_k) + \langle \nabla f(x_k), x_{k+1}-x_k \rangle + \frac{L}{2} \|x_{k+1}-x_k\|^2 \nn \\
&=& f(x_k) -\gamma_k \langle \nabla f(x_k), \tilde {g}_{_{X,k}} \rangle + \frac{L}{2}
\gamma_k^2 \|\tilde {g}_{_{X,k}}\|^2 \nn \\
&=& f(x_k) -\gamma_k \langle G_k, \tilde {g}_{_{X,k}} \rangle + \frac{L}{2} \gamma_k^2
\|\tilde {g}_{_{X,k}}\|^2 + \gamma_k \langle \delta_k, \tilde {g}_{_{X,k}} \rangle.\label{lip_grad_f}
\eeqa
So, by Lemma~\ref{proj_g_size} with $x=x_k$, $\gamma = \gamma_k$ and $g =G_k$, we obtain
\beqas
f(x_{k+1}) & \le & f(x_k) - \left[\alpha \gamma_k \|\tilde {g}_{_{X,k}}\|^2+ h(x_{k+1})-h(x_k) \right]+
\frac{L}{2} \gamma_k^2 \|\tilde {g}_{_{X,k}}\|^2  \\
& & + \gamma_k \langle \delta_k, g_{_{X,k}} \rangle +
\gamma_k \langle \delta_k, \tilde {g}_{_{X,k}}-g_{_{X,k}} \rangle,
\eeqas
where the projected gradient $g_{_{X,k}}$ is defined in \eqnok{proj_grad}. Then, from the above inequality,
\eqnok{proj_grad} and \eqnok{proj_stch_grad}, we obtain
\beqas
\Psi(x_{k+1}) &\le& \Psi(x_k) - \left(\alpha \gamma_k - \frac{L}{2} \gamma_k^2\right) \|\tilde {g}_{_{X,k}}\|^2 +
\gamma_k \langle \delta_k, g_{_{X,k}} \rangle + \gamma_k \|\delta_k\| \|\tilde {g}_{_{X,k}}-g_{_{X,k}}\| \\
&\le& \Psi(x_k) - \left(\alpha \gamma_k - \frac{L}{2} \gamma_k^2\right) \|\tilde {g}_{_{X,k}}\|^2 +
\gamma_k \langle \delta_k, g_{_{X,k}} \rangle + \frac{\gamma_k}{\alpha} \|\delta_k\|^2,
\eeqas
where the last inequality follows from Proposition~\ref{lip_proj_grad} with
$x=x_k, \gamma=\gamma_k, g_1 =G_k$ and $g_2 =\nabla f(x_k)$.
Summing up the above inequalities for $k=1,\ldots,N$ and noticing that $\gamma_k \le \alpha/L$, we obtain
\beqa \label{main_recursion_stch}
\sum_{k=1}^N \left(\alpha \gamma_k - L \gamma_k^2\right) \|\tilde {g}_{_{X,k}}\|^2
&\le& \sum_{k=1}^N \left(\alpha \gamma_k - \frac{L}{2} \gamma_k^2\right) \|\tilde {g}_{_{X,k}}\|^2 \nn \\
&\le& \Psi(x_1) - \Psi(x_{k+1}) + \sum_{k=1}^N \left\{\gamma_k \langle \delta_k, g_{_{X,k}} \rangle
+ \frac{\gamma_k}{\alpha} \|\delta_k\|^2\right\} \nn \\
&\le& \Psi(x_1) - \Psi^* + \sum_{k=1}^N \left\{\gamma_k \langle \delta_k, g_{_{X,k}} \rangle
+ \frac{\gamma_k}{ \alpha} \|\delta_k\|^2\right\}.
\eeqa
Notice that the iterate $x_k$ is a function of the history $\xi_{[k-1]}$ of the generated random process and
hence is random. By part a) of Assumption A1, we have
$\bbe[\langle \delta_k , g_{_{X,k}} \rangle |\xi_{[k-1]}] = 0$. In addition, denoting $S_j = \sum_{i=1}^{j} \delta_{k,i}$,
and noting that $\bbe[\langle S_{i-1}, \delta_{k, i}\rangle |S_{i-1}] = 0$ for
all $i = 1, \ldots, m_k$, we have
%$\bbe[\|\delta_k\|^2] \le \sigma^2 /$.
\begin{align}
\bbe[\|S_{m_k}\|^2 ]
&= \bbe \left[ \|S_{m_k -1}\|^2 + 2 \langle S_{m_k -1},
\delta_{k,m_k} \rangle + \|\delta_{k,m_k}\|^2\right] \nn\\
&= \bbe[\|S_{m_k -1}\|^2] + \bbe [\|\delta_{k,m_k}\|^2]
= \ldots =  \sum_{i=1}^{m_k} \|\delta_{k,i}\|^2,\nn
%\bbe
%\|\frac{1}{m_k}\sum_{i=1}^{m_k} \left[G(x_k,\xi_{k,i}) - \nabla f(x_k)\right]\|^2 \le \frac{\sigma^2}{m_k}.
\end{align}
which, in view of \eqnok{def_Gk} and Assumption A1.b), then implies that
\beq \label{dec_sigma0}
\bbe[\|\delta_k\|^2 ] = \frac{1}{m_k^2} \bbe[\|S_{m_k}\|^2 ] = \frac{1}{m_k^2}  \sum_{i=1}^{m_k}
\bbe[\|\delta_{k,i}\|^2] \le \frac{\sigma^2}{m_k}.
\eeq
With these observations, now taking expectations with respect to $\xi_{[N]}$
on both sides of \eqnok{main_recursion_stch}, we get
\[
\sum_{k=1}^N \left(\alpha \gamma_k - L\gamma_k^2\right) \bbe_{\xi_{[N]}} \|\tilde {g}_{_{X,k}}\|^2 \le
\Psi(x_1) - \Psi^* + (\sigma^2/ \alpha) \sum_{k=1}^N (\gamma_k/ m_k).
\]
Then, since $ \sum_{k=1}^N \left(\alpha \gamma_k - L \gamma_k^2 \right)>0$ by our assumption,
dividing both sides of the above inequality by $\sum_{k=1}^N \left(\alpha \gamma_k - L \gamma_k^2 \right)$ and
noticing that
\[
\bbe[\|\tilde{g}_{_{X,R}}\|^2] = \bbe_{R, \xi_{[N]}}[\|\tilde{g}_{_{X,R}}\|^2] =
\frac{\sum_{k=1}^N \left(\alpha \gamma_k - L \gamma_k^2\right) \bbe_{\xi_{[N]}}\|\tilde{g}_{_{X,k}}\|^2}
{\sum_{k=1}^N \left(\alpha \gamma_k - L \gamma_k^2\right)},
\]
we have  \eqnok{main_cnvg_stch} holds.

We now show part (b) of the theorem. By Lemma~\ref{comp_proj_sol} with
$x=x_k, \gamma=\gamma_k, g=G_k$ and $u=x^*$, we have
\[
\langle G_k, x_{k+1} \rangle + h(x_{k+1})+\frac{1}{\gamma_k} V(x_{k+1}, x_k)
\le \langle G_k, x^* \rangle +h(x^*)+ \frac{1}{\gamma_k} [V(x^*, x_k) - V(x^*, x_{k+1})],\nn
\]
which together with \eqnok{smooth} and definition of $\delta_k$ give
\beqa
 & & f(x_{k+1})+\langle \nabla f(x_k)+\delta_k, x_{k+1} \rangle + h(x_{k+1})+ \frac{1}{\gamma_k} V(x_{k+1}, x_k) \nn \\
&\le & f(x_k)+\langle \nabla f(x_k), x_{k+1}-x_k \rangle + \frac{L}{2} \|x_{k+1}-x_k\|^2 +
\langle \nabla f(x_k)+\delta_k, x^* \rangle + h(x^*) \nn \\
 & & + \frac{1}{\gamma_k} [V(x^*, x_k) - V(x^*, x_{k+1})].\nn
\eeqa
Simplifying the above inequality, we have
\beqa
\Psi(x_{k+1})
&\le& f(x_k)+\langle \nabla f(x_k), x^*-x_k \rangle +h(x^*)+\langle \delta_k, x^* -x_{k+1}\rangle
+\frac{L}{2} \|x_{k+1}-x_k\|^2  \nn \\
& & - \frac{1}{\gamma_k} V(x_{k+1}, x_k)+\frac{1}{\gamma_k} [V(x^*, x_k) - V(x^*, x_{k+1})].\nn
\eeqa
Then, it follows from the convexity of $f$, \eqnok{strg_cnvx} and \eqnok{prox_fun} that
\beqa
\Psi(x_{k+1})
&\le& f(x^*)+h(x^*) +\langle \delta_k, x^* -x_{k+1}\rangle +\left(\frac{L}{2}-\frac{\alpha}{2 \gamma_k}\right)
\|x_{k+1}-x_k\|^2  \nn \\
& & +\frac{1}{\gamma_k} [V(x^*, x_k) - V(x^*, x_{k+1})]\nn \\
&=& \Psi(x^*)+\langle \delta_k, x^* -x_k\rangle +\langle \delta_k, x_k -x_{k+1}\rangle +
\frac{L \gamma_k -\alpha}{2 \gamma_k} \|x_{k+1}-x_k\|^2 \nn \\
& & +\frac{1}{\gamma_k} [V(x^*, x_k) - V(x^*, x_{k+1})]\nn \\
&\le& \Psi(x^*)+\langle \delta_k, x^* -x_k\rangle +\frac{\gamma_k}{2 (\alpha- L \gamma_k )} \|\delta_k\|^2
+\frac{1}{\gamma_k} [V(x^*, x_k) - V(x^*, x_{k+1})],\nn
\eeqa
where the last inequality follows from Young's inequality.
Noticing  $\gamma_k \le \alpha/L$, multiplying both sides of the above inequality by $(\alpha \gamma_k- L \gamma_k^2)$
and summing them up for $k=1,\ldots,N$, we obtain
\beqa
\sum_{k=1}^N (\alpha \gamma_k - L \gamma_k^2 ) \left[\Psi(x_{k+1}) - \Psi(x^*)\right]
&\le& \sum_{k=1}^N (\alpha \gamma_k - L \gamma_k^2 ) \langle \delta_k, x^* -x_k\rangle + \sum_{k=1}^N \frac{\gamma_k^2}{2}\|\delta_k\|^2 \nn \\
&+&\sum_{k=1}^N(\alpha- L \gamma_k)\left[V(x^*, x_k) - V(x^*, x_{k+1})\right].\label{main_recursion_stch1}
\eeqa
Now, if the increasing stepsize condition \eqnok{incr_stepsize} is satisfied, we have from $V(x^*, x_{N+1}) \ge 0$ that
\beqa
 & & \sum_{k=1}^N(\alpha- L \gamma_k)\left[V(x^*, x_k) - V(x^*, x_{k+1})\right] \nn \\
&=& (\alpha- L \gamma_1)V(x^*, x_1)+\sum_{k=2}^N(\alpha- L \gamma_k)V(x^*, x_k)- \sum_{k=1}^N
(\alpha- L \gamma_k)V(x^*, x_{k+1}) \nn \\
&\le& (\alpha- L \gamma_1)V(x^*, x_1)+\sum_{k=2}^N(\alpha- L \gamma_{k-1})V(x^*, x_k)-
\sum_{k=1}^N (\alpha- L \gamma_k)V(x^*, x_{k+1}) \nn \\
&=&(\alpha- L \gamma_1)V(x^*, x_1)-(\alpha- L \gamma_N)V(x^*, x_{N+1}) \nn \\
&\le& (\alpha- L \gamma_1)V(x^*, x_1). \nn
\eeqa
Taking expectation on both sides of \eqnok{main_recursion_stch1} with respect to $\xi_{[N]}$,
again using the observations that $\bbe[\|\delta_k^2\|] \le \sigma^2$ and
$\bbe[\langle \delta_k , g_{_{X,k}} \rangle |\xi_{[k-1]}] = 0$, then it follows from the above inequality that
\[
\sum_{k=1}^N (\alpha \gamma_k - L \gamma_k^2 ) \bbe_{\xi_{[N]}} \left[\Psi(x_{k+1}) - \Psi(x^*)\right]
\le (\alpha- L \gamma_1)V(x^*,x_1) + \frac{\sigma^2}{2} \sum_{k=1}^N (\gamma_k^2/ m_k).
\]
Finally, \eqnok{main_cnvg_stch_cvx1} follows from the above inequality and the arguments similar to the
proof in part (a). Now, if the decreasing stepsize condition \eqnok{dec_stepsize} is satisfied,
we have from the definition $\bar V(x^*):= \max_{u \in X} V(x^*,u) \ge 0$ and  $V(x^*, x_{N+1}) \ge 0$ that
\beqa
 & & \sum_{k=1}^N(\alpha- L \gamma_k)\left[V(x^*, x_k) - V(x^*, x_{k+1})\right] \nn \\
& = & (\alpha- L \gamma_1) V(x^*,x_1) + L \sum_{k=1}^{N-1} (\gamma_k - \gamma_{k+1})
V(x^*, x_{k+1}) - (\alpha- L \gamma_N) V(x^*, x_{N+1}) \nn \\
&\le& (\alpha- L \gamma_1) \bar{V}(x^*) + L \sum_{k=1}^{N-1}  (\gamma_k - \gamma_{k+1})
\bar{V}(x^*) - (\alpha- L \gamma_N)V(x^*, x_{N+1}) \nn \\
&\le & (\alpha- L \gamma_N)\bar{V}(x^*), \nn
\eeqa
which together with \eqnok{main_recursion_stch1} and similar arguments used above would give \eqnok{main_cnvg_stch_cvx2}.
\end{proof}

\vgap

A few remarks about Theorem~\ref{main_theorem_stch} are in place.
Firstly, if $f$ is convex and the batch sizes $m_k =1$, then by properly choosing
the stepsizes $\{\gamma_k\}$ (e.g., $\gamma_k = \C{O} (1/\sqrt{k})$ for $k$ large),
we can still guarantee a nearly optimal rate of convergence for the RSPG algorithm (see
\eqnok{main_cnvg_stch_cvx1} or \eqnok{main_cnvg_stch_cvx2}, and \cite{NJLS09-1,Lan10-3}).
%by using \eqnok{main_cnvg_stch_cvx1} or \eqnok{main_cnvg_stch_cvx2} 
However, if $f$ is possibly nonconvex and $m_k =1$, then the RHS of \eqnok{main_cnvg_stch}
is bounded from below by
\[
 \frac{L D_{\Psi}^2 + (\sigma^2/\alpha) {\sum_{k=1}^N \gamma_k}}{{\sum_{k=1}^N (\alpha \gamma_k - L\gamma_k^2)}}
 \ge \frac{\sigma^2}{\alpha^2},
\]
which does not necessarily guarantee the converge of the RSPG algorithm, no matter how the stepsizes
$\{\gamma_k\}$ are specified. This is exactly the reason why we
consider taking multiple samples $G(x_k, \xi_{k, i})$,
$i = 1, \ldots, m_k$, for some $m_k > 1$ at each iteration
of the RSPG method.

Secondly, from \eqnok{main_recursion_stch} in the proof of Theorem~\ref{main_theorem_stch}, we see that the stepsize
policies can be further relaxed to get a similar result as \eqnok{main_cnvg_stch}. More specifically, we can
have the following corollary.
\begin{corollary}
Suppose that the stepsizes $\{\gamma_k\}$ in the RSPG algorithm are chosen such that $ 0 < \gamma_k \le 2 \alpha/ L$
with $\gamma_k < 2 \alpha/L$ for at least one $k$, and the probability mass function $P_R$ are chosen such that
for any $k = 1, \ldots, N$
\beq \label{prob_fun-relax}
P_R(k) := \Prob\{R=k\} = \frac{\alpha \gamma_k- L\gamma_k^2/2}
{{\sum_{k=1}^N (\alpha \gamma_k- L \gamma_k^2/2)}}.
\eeq
Then, under Assumption A1, we have
\beq \label{cnvg_stch-relax}
\bbe[\|\tilde{g}_{_{X,R}}\|^2]
\le \frac{L D_{\Psi}^2 + (\sigma^2/\alpha) {\sum_{k=1}^N (\gamma_k/ m_k)}}
{{\sum_{k=1}^N (\alpha \gamma_k - L\gamma_k^2/2)}},
\eeq
where the expectation is taken with respect to $R$ and $\xi_{[N]} := (\xi_1,...,\xi_N)$.
\end{corollary}

\vgap

Based on the Theorem~\ref{main_theorem_stch}, we can establish the
following complexity results of the RSPG algorithm with proper selection of
stepsizes $\{\gamma_k\}$  and batch sizes $\{ m_k\}$ at each iteration.

\begin{corollary} \label{RSPG_m0}
Suppose that in the RSPG algorithm the stepsizes $\gamma_k = \alpha/(2L)$ for all $k=1,\ldots,N$,
and the probability mass function $P_R$ are chosen as \eqnok{prob_fun}.
Also assume that the batch sizes $m_k = m$, $k = 1, \ldots, N$, for some $m \ge 1$.
Then under Assumption A1, we have
\beq
\bbe[\|g_{_{X,R}}\|^2]  \le
\frac{8 L^2 D_{\Psi}^2}{\alpha^2 N} + \frac{6 \sigma^2}{\alpha^2 m} \quad \mbox{and} \quad
\bbe[\|\tilde g_{_{X,R}}\|^2]  \le
\frac{4L^2 D_{\Psi}^2}{\alpha^2 N} + \frac{2\sigma^2}{\alpha^2 m},
 \label{proj_grad_cnvg_stch}
\eeq
where $g_{_{X,R}}$ and $\tilde g_{_{X,R}}$  are defined in \eqnok{proj_grad}
and \eqnok{proj_stch_grad}, respectively.
%\beq \label{proj_stch_grad2}
%\bar g_{_{X,R}} = P_X(x_k, \bar G_m(x_R), \gamma_R),
%\eeq
%and $g_{_{X,R}}$ is defined in \eqnok{proj_grad}.
If, in addition, $f$ in the problem \eqnok{NLP} is convex with an optimal solution $x^*$, then
\beq \label{proj_stch_grad_cnvg_stch_cvx}
\bbe \left[\Psi(x_R) - \Psi(x^*)\right]
\le \frac{2LV(x^*,x_1)}{N\alpha}+\frac{\sigma^2}{2Lm}.
\eeq
\end{corollary}

\begin{proof}
%First note that if we take $m$ samples of the stochastic gradients at each iteration of the RSPG method,
%then under Assumption A1, we have from from Lemma~\ref{marting}.a with $\zeta_i = G(x_k,\xi_{k,i}) - \nabla f(x_k)$ that
%\beq \label{dec_sigma}
%\bbe \left[ \|\bar G_m(x_k) - \nabla f(x_k)\|^2 \right] = \bbe
%\|\frac{1}{m}\sum_{i=1}^m \left[G(x_k,\xi_{k,i}) - \nabla f(x_k)\right]\|^2 \le \frac{\sigma^2}{m}.
%\eeq
%Hence, by using $\bar{G}(x_k)$ instead of $G_k$ in \eqnok{update_RSPG}, similar to
By \eqnok{main_cnvg_stch}, we have
\[
\bbe[\|\tilde g_{_{X,R}}\|^2] \le \frac{ L D_{\Psi}^2 + \frac{\sigma^2}{m \alpha}
{\sum_{k=1}^N \gamma_k}}{{\sum_{k=1}^N (\alpha \gamma_k - L\gamma_k^2)}},
\]
which together with $\gamma_k = \alpha/(2L)$ for all $k=1,\ldots,N$ imply that
\[
\bbe[\|\tilde g_{_{X,R}}\|^2] = \frac{L D_{\Psi}^2 + \frac{\sigma^2 N}{2mL}}{\frac{N \alpha^2}{4L}}
= \frac{4L^2 D_{\Psi}^2}{N \alpha^2} + \frac{2\sigma^2}{m\alpha^2}.
\]
Then, by Proposition~\ref{lip_proj_grad} with $x=x_R, \gamma=\gamma_R, g_1 = \nabla f(x_R), g_2= G_k$,
we have from the above inequality and \eqnok{dec_sigma0} that
\beqa
\bbe[\|g_{_{X,R}}\|^2] &\le& 2 \bbe[\|\tilde g_{_{X,R}}\|^2] + 2\bbe[\|g_{_{X,R}}-\tilde g_{_{X,R}}\|^2] \nn \\
&\le& 2 \left( \frac{4L^2 D_{\Psi}^2}{N \alpha^2}+\frac{2\sigma^2}{\alpha^2 m}\right) +
\frac{2}{\alpha^2}\bbe \left[ \|G_k-\nabla f(x_R)\|^2 \right]  \nn \\
&\le& \frac{8L^2 D_{\Psi}^2}{N \alpha^2}+ \frac{6\sigma^2}{\alpha^2 m}. \nn
\eeqa

Moreover, since $\gamma_k = \alpha/(2L)$ for all $k=1,\ldots,N$, the stepsize conditions \eqnok{incr_stepsize}
are satisfied. Hence, if the problem is convex, \eqnok{proj_stch_grad_cnvg_stch_cvx}
can be derived in a similar way as \eqnok{main_cnvg_stch_cvx1}.
\end{proof}

\vgap

Note that all the bounds in the above corollary depend on $m$.
%For example, if $m=1$, then the second terms in bounds
%\eqnok{proj_grad_cnvg_stch} and \eqnok{proj_stch_grad_cnvg_stch_cvx} are fixed and the algorithm actually does not converge.
Indeed, if $m$ is set to some fixed positive integer constant, then the second terms in the above results will
always majorize the first terms when $N$ is sufficiently large. Hence, the appropriate choice of $m$ should be
balanced with the number of iterations $N$, which would eventually depend on the total computational budget given
by the user. The following corollary shows an appropriate choice of $m$ depending on the total number of calls
to the ${\cal SFO}$.
\begin{corollary} \label{RSPG_m}
Suppose that all the conditions in Corollary~\ref{RSPG_m0} are satisfied.
Given a fixed total  number of calls $\bar N$ to the ${\cal SFO}$, if
the number of calls to the ${\cal SFO}$ (number of samples) at each iteration of the RSPG algorithm is
\beq \label{def_m}
m =\left \lceil \min \left\{ \max \left\{1, \frac{\sigma \sqrt{6 \bar N}}{4 L \tilde D} \right\}, \bar N
\right\} \right \rceil,
\eeq
for some $\tilde D >0$, then we have $ (\alpha^2/L) \; \bbe[\|g_{_{X,R}}\|^2] \le {\cal B}_{\bar N}$, where
\beq \label{proj_stch_grad_cnvg_stch_m}
{\cal B}_{\bar N}:=\frac{16 L D_{\Psi}^2}{\bar N}+
\frac{4 \sqrt 6 \sigma}{\sqrt{\bar N}} \left( \frac{D_{\Psi}^2}{\tilde D} +
\tilde D \max \left\{1,\frac{\sqrt 6 \sigma}{4 L \tilde D \sqrt{\bar N} }  \right\} \right).
\eeq

If, in addition, $f$ in problem \eqnok{NLP} is convex, then
$\bbe[\Psi(x_R)-\Psi(x^*)] \le {\cal C}_{\bar N}$, where $x^*$ is an optimal solution and
\beq \label{proj_stch_grad_cnvg_stch_cvx_m}
{\cal C}_{\bar N}:=\frac{4 L V(x^*,x_1)}{\alpha \bar N} + \frac{\sqrt 6 \sigma}{\alpha \sqrt{\bar N}}
\left( \frac{V(x^*,x_1)}{\tilde D} +
\frac{\alpha \tilde D}{3} \max \left\{1,\frac{\sqrt 6 \sigma}{4 L \tilde D \sqrt{\bar N} } \right\} \right).
\eeq
\end{corollary}
\begin{proof}
Given the total number of calls to the stochastic first-order oracle $\bar N$ and
the number $m$ of calls to the ${\cal SFO}$ at each iteration, the RSPG algorithm can
perform at most $N = \lfloor \bar N/m \rfloor$ iterations. Obviously, $N \ge \bar N/ (2m)$.
With this observation and \eqnok{proj_grad_cnvg_stch}, we have
\beqa
\bbe[\|g_{_{X,R}}\|^2] &\le& \frac{16 m L^2 D_{\Psi}^2}{\alpha^2 \bar N} + \frac{6 \sigma^2}{\alpha^2 m} \nn \\
&\le& \frac{16 L^2 D_{\Psi}^2}{\alpha^2 \bar N} \left(1+\frac{\sigma \sqrt{6 \bar N}}{4  L \tilde D}\right)
+ \max \left\{ \frac{4 \sqrt 6 L \tilde D \sigma}{\alpha^2 \sqrt{\bar N}}, \frac{6 \sigma^2}{\alpha^2 \bar N} \right\} \nn \\
&=& \frac{16 L^2 D_{\Psi}^2}{ \alpha^2 \bar N}+ \frac{4 \sqrt 6 L \sigma}{\alpha^2 \sqrt{\bar N}}
\left( \frac{D_{\Psi}^2}{\tilde D} + \tilde D \max \left\{1,\frac{\sqrt 6 \sigma}{4 L \tilde D \sqrt{\bar N} }
 \right\} \right),
\eeqa
which gives \eqnok{proj_stch_grad_cnvg_stch_m}.
The bound \eqnok{proj_stch_grad_cnvg_stch_cvx_m} can be obtained in a similar way.
\end{proof}

\vgap

We now would like add a few remarks about the above results in Corollary~\ref{RSPG_m}.
Firstly, although we use the constant value for $m_k=m$ at each iteration, one can also choose it adaptively
during the execution of the RSPG algorithm while monitoring the convergence. For example, in practice
$m_k$ could adaptively depend on $\sigma_k^2 :=\bbe \left[ \|G(x_k, \xi_k) - \nabla f(x_k)\|^2 \right]$.
Secondly, we need to specify the parameter $\tilde D$ in \eqnok{def_m}. It can be seen
from \eqnok{proj_stch_grad_cnvg_stch_m} and \eqnok{proj_stch_grad_cnvg_stch_cvx_m} that when $\bar N$ is relatively
large such that
\beq \label{large-barN}
 \max \left\{1,\sqrt 6 \sigma/(4 L \tilde D \sqrt{\bar N})  \right\} = 1, \quad \mbox{i.e.,} \quad
\bar N \ge 3 \sigma^2/(8 L^2 \tilde{D}^2),
\eeq
an optimal choice of $\tilde D$ would be $D_{\Psi}$ and $\sqrt{3V(x^*,x_1)/\alpha}$ for solving nonconvex and
convex SP problems, respectively. With this selection of $\tilde D$, the bounds in \eqnok{proj_stch_grad_cnvg_stch_m}
and \eqnok{proj_stch_grad_cnvg_stch_cvx_m}, respectively, reduce to
\beq \label{nocvx_stch_m}
\frac{\alpha^2}{L} \bbe[\|g_{_{X,R}}\|^2]
\le \frac{16 L D_{\Psi}^2}{\bar N}+ \frac{8 \sqrt 6 D_{\Psi} \sigma}{\sqrt{\bar N}}
\eeq
and
\beq \label{cvx_stch_m}
\bbe[\Psi(x^*)-\Psi(x_1)] \le \frac{4 L V(x^*,x_1)}{\alpha \bar N} +
\frac{2 \sqrt{2 V(x^*,x_1)} \sigma}{\sqrt{\alpha \bar N}}.
\eeq
Thirdly, the stepsize policy in Corollary~\ref{RSPG_m0} and the probability mass function \eqnok{prob_fun} together
with the number of samples \eqnok{def_m} at each iteration of the RSPG algorithm provide a unified strategy for solving
both convex and nonconvex SP problems. In particular, the RSPG algorithm exhibits a nearly optimal
rate of convergence for
solving smooth convex SP problems, since the second term in \eqnok{cvx_stch_m} is unimprovable (see e.g., \cite{nemyud:83}),
while the first term in \eqnok{cvx_stch_m} can be considerably improved~\cite{Lan10-3}.
%Moreover, while the first term in \eqnok{cvx_stch_m} can be considerably improved for convex problems, the second term
%in \eqnok{nocvx_stch_m} is shown in \cite{CarGouToi10-1} to be unimprovable under first-order information of $f$.

\subsection{A two-phase randomized stochastic  projected gradient method} \label{sec_2RSPG}
In the previous subsection, we present the expected complexity results over many runs of the RSPG algorithm.
Indeed, we are also interested in the performance of a single run of RSPG. In particular, we want to establish
the complexity results for finding an {\sl $(\epsilon, \Lambda)$-solution} of the problem \eqnok{NLP}, i.e.,
a point $x \in X$ satisfying $\Prob\{\|g_{_X}(x)\|^2 \le \epsilon\} \ge 1 - \Lambda$, for some $\epsilon >0$ and $\Lambda \in (0,1)$.
Noticing that by the Markov's inequality and \eqnok{proj_stch_grad_cnvg_stch_m}, we can directly have
\beq \label{nocvx_prob}
\Prob \left\{
\|g_{_{X,R}}\|^2 \ge  \frac{\lambda L {\cal B}_{\bar N}}{\alpha^2} \right\}
\le \frac{1}{\lambda}, \qquad \mbox{for any } \lambda > 0.
\eeq
This implies that the total number of calls to the ${\cal SFO}$ performed by the RSPG algorithm
for finding an {\sl $(\epsilon, \Lambda)$-solution}, after disregarding a few constant factors, can be bounded by
\beq \label{nocvx_prob1}
{\cal O} \left\{
\frac{1}{\Lambda \epsilon} +
\frac{\sigma^2}{\Lambda^2 \epsilon^2}\right\}.
\eeq
In this subsection, we present a approach to improve the dependence of the above bound on $\Lambda$.
More specifically, we propose a variant of the RSPG algorithm which has two phases: an optimization phase and
a post-optimization phase. The optimization phase consists of independent single runs of the RSPG algorithm to
generate a list of candidate solutions, and in the post-optimization phase, we choose a solution $x^*$ from these
candidate solutions generated by the optimization phase.
For the sake of simplicity, we assume throughout this subsection that the
norm $\|\cdot\|$ in $\bbr^n$ is the standard Euclidean norm.

\vskip 0.1cm

\noindent {\bf A two phase RSPG (2-RSPG) algorithm}
\begin{itemize}
\item [] {\bf Input:} Given initial point $x_1 \in X$, number of runs $S$,
total $\bar N$ of calls to the ${\cal SFO}$ in each run of the RSPG algorithm,
and sample size $T$ in the post-optimization phase.
\item [] {\bf Optimization phase:}
\item [] \hspace{0.1in} For $s = 1, \ldots, S$
\begin{itemize}
\item [] Call the RSPG algorithm with initial point $x_1$,
iteration limit $N=\lfloor \bar N/m \rfloor$ with $m$ given by \eqnok{def_m},
stepsizes $\gamma_k = \alpha/(2L)$ for $k=1,\ldots,N$,
batch sizes $m_k = m$, and
probability mass function $P_{R}$ in \eqnok{prob_fun}.
\end{itemize}
\item [] \hspace{0.1in} Let $\bar x_s = x_{R_s}$, $s=1,\ldots,S$, be the outputs of this phase.
\item [] {\bf Post-optimization phase:}
\item[] \hspace{0.1in} Choose a solution $\bar x^*$ from the candidate
list $\{\bar x_1, \ldots, \bar x_S\}$ such that
\beq \label{post_opt}
\|\bar g_{_X}(\bar x^*)\| = \min_{s = 1, \ldots, S} \|\bar g_{_X}(\bar x_s)\|, \quad
\bar g_{_X}(\bar x_s) := P_X(\bar x_s,\bar G_T(\bar x_s), \gamma_{R_s}),
\eeq
\hspace{0.1in} where $\bar G_T(x)=\frac{1}{T} \sum_{k=1}^T G(x, \xi_k)$ and $P_X(x,g,\gamma)$ is defined in \eqnok{proj_g}.
\item [] {\bf Output:} $\bar x^*$.
\end{itemize}

\vgap

In the 2-RSPG algorithm, the total number of calls of ${\cal SFO}$ in the optimization phase and post-optimization phase
is bounded by $S \times \bar N$ and $S \times T$, respectively. In the next theorem, we provide certain bounds of
$S$, $\bar N$ and $T$ for finding an {\sl $(\epsilon, \Lambda)$-solution} of problem \eqnok{NLP}.

We need the following well-known large deviation theorem of vector-valued martingales
to derive the large deviation results of the 2-RSPG algorithm
(see \cite{jn08-1} for a general result using possibly non-Euclidean norm).

\begin{lemma} \label{marting}
Assume that we are given a polish space with Borel probability measure $\mu$ and
a sequence of ${\cal F}_0 = \{\emptyset,\Omega\} \subseteq {\cal F}_1
\subseteq {\cal F}_2 \subseteq \ldots$ of $\sigma$-sub-algebras of Borel
$\sigma$-algebra of $\Omega$.
Let $\zeta_i \in \bbr^n$, $i = 1, \ldots, \infty$, be a martingale-difference sequence
of Borel functions on $\Omega$ such that $\zeta_i$ is ${\cal F}_i$  measurable and
$\bbe[\zeta_i| i-1] =0$, where $\bbe[\cdot|i]$, $i = 1, 2, \ldots$, denotes the
conditional expectation w.r.t. ${\cal F}_i$ and $\bbe \equiv \bbe[\cdot|0]$ is the expectation w.r.t.
$\mu$.
\begin{itemize}
\item [a)] If $\bbe[\|\zeta_i\|^2] \le \sigma_i^2$ for any $i \ge 1$, then
$
\bbe[\|\sum_{i=1}^N \zeta_i\|^2] \le \sum_{i=1}^N \sigma_i^2$. As a consequence, we have
\[
\forall N \ge 1, \lambda \ge 0: \Prob\left\{
\|\sum_{i=1}^N \zeta_i \|^2 \ge \lambda \sum_{i=1}^N \sigma_i^2
\right\} \le \frac{1}{\lambda};
\]
\item [b)] If $\bbe\left[\exp\left(
\|\zeta_i\|^2 /\sigma_i^{2}  \right)|i-1 \right] \le \exp(1)$ almost surely for any $i \ge 1$, then
\[
\forall N \ge 1, \lambda \ge 0: \Prob\left\{
\|\sum_{i=1}^N \zeta_i \| \ge \sqrt{2} (1+ \lambda) \sqrt{\sum_{i=1}^N \sigma_i^2}
\right\} \le \exp(-\lambda^2 /3).
\]
\end{itemize}
\end{lemma}

We are now ready to state the main convergence properties for the $2$-RSPG algorithm.
\begin{theorem} \label{2RSPG_theorem}
Under Assumption A1,  the following statements hold for the $2$-RSPG algorithm applied
to problem \eqnok{NLP}.
\begin{itemize}
\item [(a)] Let ${\cal B}_{\bar N}$ be defined in \eqnok{proj_stch_grad_cnvg_stch_m}. Then, for all $\lambda >0$
\beq \label{2RSG_conv1}
\Prob \left\{
\|g_{_X}(\bar x^*)\|^2 \ge \frac{2}{\alpha^2} \left(4 L {\cal B}_{\bar N} + \frac{3 \lambda \sigma^2}{T} \right)
\right\} \le \frac{S}{\lambda} + 2^{-S};
\eeq
\item [(b)] Let $\epsilon > 0$ and $\Lambda \in (0,1)$ be given. If the parameters $(S,\bar N,T)$ are set to
\beqa
& & \hspace{0.3in} S(\Lambda) := \left \lceil \log_2 (2/ \Lambda) \right \rceil, \label{def_S}\\
& & \hspace{0.3in} \bar N (\epsilon) := \left \lceil \max \left\{ \frac{512 L^2 D_{\Psi}^2}{\alpha^2 \epsilon},
\left[ \left( \tilde D + \frac{D_{\Psi}^2}{\tilde D}\right) \frac{ 128 \sqrt{6} L\sigma}{\alpha^2 \epsilon} \right]^2,
\frac{3 \sigma^2}{8 L^2 \tilde{D}^2 } \right\} \right \rceil, \label{def_N} \\
& & \hspace{0.3in} T(\epsilon,\Lambda) := \left \lceil \frac{24 S(\Lambda) \sigma^2}{\alpha^2 \Lambda \epsilon} \right \rceil,\label{def_Tbar}
\eeqa
then the $2$-RSPG algorithm computes an $(\epsilon, \Lambda)$-solution of the problem \eqnok{NLP}
after taking at most
\beq \label{bnd_compl}
S(\Lambda) \, \left[ \bar N(\epsilon) + T(\epsilon, \Lambda)\right]
\eeq
calls of the stochastic first order oracle.
\end{itemize}
\end{theorem}
\begin{proof}
We first show part (a). Let $g_{_X}(\bar x_s) = P_X(\bar x_s,\nabla f(\bar x_s), \gamma_{R_s})$. Then, it
follows from the definition of $\bar x^*$ in \eqnok{post_opt} that
\beqas
 \|\bar g_{_X}(\bar x^*)\|^2 &=& \min_{s=1,\ldots,S} \|\bar g_{_X}(\bar x_s)\|^2
= \min_{s=1,\ldots,S} \| g_{_X}(\bar x_s) + \bar g_{_X}(\bar x_s) - g_{_X}(\bar x_s)\|^2 \\
& \le & \min_{s=1, \ldots,S} \left\{2 \|g_{_X}(\bar x_s)\|^2 + 2
\|\bar g_{_X}(\bar x_s) - g_{_X}(\bar x_s) \|^2 \right\} \\
 &\le& 2 \min_{s=1, \ldots,S} \|g_{_X}(\bar x_s)\|^2 +
2 \max_{s=1,\ldots,S} \|\bar g_{_X}(\bar x_s) - g_{_X}(\bar x_s) \|^2,
\eeqas
which implies that
\beqa
\|g_{_X}(\bar x^*)\|^2 &\le& 2 \|\bar g_{_X}(\bar x^*)\|^2 + 2 \|g_{_X}(\bar x^*) - \bar g_{_X}(\bar x^*)\|^2 \nn \\
&\le& 4 \min_{s=1, \ldots,S} \|g_{_X}(\bar x_s)\|^2 + 4 \max_{s=1,\ldots,S} \|\bar g_{_X}(\bar x_s) - g_{_X}(\bar x_s) \|^2
+ 2 \|g_{_X}(\bar x^*) - \bar g_{_X}(\bar x^*)\|^2  \nn \\
&\le& 4 \min_{s=1, \ldots,S} \|g_{_X}(\bar x_s)\|^2 + 6 \max_{s=1,\ldots,S} \|\bar g_{_X}(\bar x_s) - g_{_X}(\bar x_s) \|^2.
 \label{rel_post}
\eeqa
We now provide certain probabilistic bounds to the two terms in the
right hand side of the above inequality. Firstly, from the fact
that $\bar x_s$, $1\le s \le S$, are independent and \eqnok{nocvx_prob} (with $\lambda = 2$), we have
\beq \label{opt_phase_result}
\hspace{0.3in}  \Prob \left\{ \min_{s \in [1,S]} \|g_{_X}(\bar x_s)\|^2
\ge \frac{2 L {\cal B}_{\bar N}}{\alpha^2} \right\}
  =  \prod_{s=1}^S \Prob\left\{\|g_{_X}(\bar x_s)\|^2 \ge
\frac{2 L {\cal B}_{\bar N}}{\alpha^2} \right\} \le 2^{-S}.
\eeq
Moreover, denoting $\delta_{s,k} = G(\bar x_s, \xi_k) - \nabla f(\bar x_s)$, $k = 1, \ldots, T$, by
Proposition~\ref{lip_proj_grad} with $x=\bar x_s, \gamma=\gamma_{R_{s}}, g_1 =\bar G_T(\bar x_s),
g_2 =\nabla f(\bar x_s)$, we have
\beq \label{error-bar}
\|\bar g_{_X}(\bar x_s) - g_{_X}(\bar x_s)\| \le \frac{1}{\alpha} \|\sum_{k=1}^T
 \delta_{s,k} / T\|.
\eeq
From the above inequality, Assumption A1 and Lemma~\ref{marting}.a),
for any $\lambda > 0$ and any $s = 1, \ldots, S$, we have
\[
\Prob\left\{ \| \bar g_{_X}(\bar x_s) - g_{_X}(\bar x_s)\|^2 \ge  \frac{\lambda \sigma^2}{\alpha^2 T} \right\}
 \le \Prob\left\{ \|\sum_{k=1}^T \delta_{s,k} \|^2 \ge  \lambda T \sigma^2 \right\}
\le \frac{1}{\lambda},
\]
which implies
\beq \label{closeness1}
\Prob \left\{ \max_{s=1,\ldots,S} \|\bar g_{_X}(\bar x_s) - g_{_X}(\bar x_s)\|^2 \ge \frac{\lambda \sigma^2}{\alpha^2T}
 \right\} \le \frac{S}{\lambda}.
\eeq
Then, the conclusion \eqnok{2RSG_conv1} follows from \eqnok{rel_post}, \eqnok{opt_phase_result} and
\eqnok{closeness1}.

We now show part (b). With the settings in part (b), it is easy to count the total number of calls of the ${\cal SFO}$
in the 2-RSPG algorithm is bounded up by \eqnok{bnd_compl}. Hence, we only need to show that the
$\bar x^*$ returned by the 2-RSPG algorithm is indeed an $(\epsilon, \Lambda)$-solution of the problem \eqnok{NLP}.
With the choice of $\bar{N}(\epsilon)$ in \eqnok{def_N}, we can see that \eqnok{large-barN} holds.
So, we have from \eqnok{proj_stch_grad_cnvg_stch_m} and \eqnok{def_N} that
\[
{\cal B}_{\bar N(\epsilon)} = \frac{16 L D_{\Psi}^2}{\bar N(\epsilon)} +
\frac{4 \sqrt{6} \sigma}{\sqrt{\bar N(\epsilon)}} \left( \tilde D + \frac{D_{\Psi}^2}{\tilde D}\right)
\le \frac{\alpha^2 \epsilon}{32L} + \frac{\alpha^2 \epsilon}{32L} = \frac{\alpha^2 \epsilon}{16L}.
\]
By the above inequality and \eqnok{def_Tbar}, setting $\lambda = 2 S/\Lambda$ in \eqnok{2RSG_conv1},
we have
\[
\frac{8 L {\cal B}_{\bar N(\epsilon)}}{\alpha^2} + \frac{6 \lambda \sigma^2}{\alpha^2 T(\epsilon, \Lambda)}
\le \frac{\epsilon}{2} + \frac{\lambda \Lambda \epsilon}{4 S} = \epsilon,
\]
which together with  \eqnok{2RSG_conv1}, \eqnok{def_S}
and $\lambda= 2 S/\Lambda$ imply
\[
\Prob\left\{\|g_{_X}(\bar x^*)\|^2  \ge \epsilon \right\} \le \frac{\Lambda}{2} +
2^{-S} \le \Lambda.
\]
Hence,$\bar x^*$ is an $(\epsilon, \Lambda)$-solution of the problem \eqnok{NLP}.
\end{proof}

\vgap

Now, it is interesting to compare the complexity bound in \eqnok{bnd_compl} with the one in
\eqnok{nocvx_prob1}. In view of \eqnok{def_S}, \eqnok{def_N} and \eqnok{def_Tbar},
the complexity bound in \eqnok{bnd_compl} for finding an  $(\epsilon, \Lambda)$-solution,
after discarding a few constant factors, is equivalent to
\beq \label{improved_compl}
\hspace{0.1in} {\cal O} \left\{ \frac{1}{\epsilon}
\log_2 \frac{1}{\Lambda} +
\frac{ \sigma^2}{\epsilon^2 } \log_2 \frac{1}{\Lambda}
+ \frac{ \sigma^2 }{\Lambda \epsilon} \log_2^2\frac{1}{\Lambda}
\right\}.
\eeq
When the second terms are the dominating terms in both bounds, the above bound \eqnok{improved_compl} can be
considerably smaller than the one in \eqnok{nocvx_prob1} up to a factor of
$
1 / \left[\Lambda^2 \log_2(1/\Lambda)  \right].
$

\vgap

The following theorem shows that under a certain ``light-tail" assumption:

{\bf A2:} For any $x_k \in X$, we have
\beq \label{A2}
 \bbe [\exp \{ \|G(x_k, \xi_k) - \nabla f(x)\|^2 /\sigma^2 \}] \le \exp \{1\},
\eeq
the bound \eqnok{bnd_compl} in Theorem~\ref{2RSPG_theorem} can be further improved.
\begin{corollary}
Under Assumptions A1 and A2, the following statements hold for the $2$-RSPG algorithm
applied to problem \eqnok{NLP}.
\begin{itemize}
\item [(a)] Let ${\cal B}_{\bar N}$ is defined in \eqnok{proj_stch_grad_cnvg_stch_m}. Then, for all $\lambda >0$
\beq \label{2RSPG_conv1_1}
\hspace{0.3in} \Prob\left\{
\|g_{_X}(\bar x^*)\|^2 \ge \left[\frac{8 L {\cal B}_{\bar N}}{\alpha^2} +
\frac{12 (1 + \lambda)^2 \sigma^2}{T \alpha^2} \right] \right\} \le S \exp(-\frac{\lambda^2}{3}) + 2^{-S};
\eeq
\item [(b)] Let $\epsilon > 0$ and $\Lambda \in (0,1)$ be given.
If $S$ and $\bar N$ are set to $S(\Lambda)$ and $\bar N(\epsilon)$ as in
\eqnok{def_S} and \eqnok{def_N}, respectively, and
the sample size $T$ is set to
\beq \label{def_Tbar1}
T'(\epsilon, \Lambda):= \frac{24 \sigma^2}{\alpha^2 \epsilon}\left[1 + \left(3 \log_2 \frac{2 S(\Lambda)}
{\Lambda} \right)^\frac{1}{2} \right]^2,
\eeq
then the $2$-RSPG algorithm can compute an $(\epsilon, \Lambda)$-solution of the problem \eqnok{NLP}
after taking at most
\beq \label{bnd_comp2}
S(\Lambda) \left[ \bar N(\epsilon) + T'(\epsilon, \Lambda)\right]
\eeq
calls to the stochastic first-order oracle.
\end{itemize}
\end{corollary}
\begin{proof}
We only give a sketch of the proof for part (a). The proof of part (b) follows from part (a) and
similar arguments for proving (b) part of Theorem~\ref{2RSPG_theorem}.
Now, denoting $\delta_{s,k} = G(\bar x_s, \xi_k) - \nabla f(\bar x_s)$, $k = 1, \ldots, T$,
again by Proposition~\ref{lip_proj_grad}, we have \eqnok{error-bar} holds.
Then, by Assumption A2 and Lemma~\ref{marting}.b), for any $\lambda > 0$ and any $s = 1, \ldots, S$, we have
\beqa
 && \Prob\left\{ \|\bar g_{_X}(\bar x_s) - g_{_X}(\bar x_s)\|^2 \ge (1+\lambda)^2 \frac{2\sigma^2}{ \alpha^2 T} \right\}
 \nn \\
& \le & \Prob\left\{ \|\sum_{k=1}^T \delta_{s,k} \| \ge  \sqrt{2T} (1+\lambda) \sigma \right\}
\le \exp(-\frac{\lambda^2}{3}), \nn
\eeqa
which implies that  for any $\lambda > 0$
\beq \label{closeness1_p}
\Prob\left\{
\max_{s=1,\ldots,S} \|\bar g_{_X}(\bar x_s) - g_{_X}(\bar x_s)\|^2 \ge (1+\lambda)^2 \frac{2\sigma^2}{\alpha^2 T}
 \right\} \le S \exp(-\frac{\lambda^2}{3}),
\eeq
Then, the conclusion \eqnok{2RSPG_conv1_1} follows from \eqnok{rel_post}, \eqnok{opt_phase_result} and
\eqnok{closeness1_p}.
\end{proof}

\vgap

In view of \eqnok{def_S}, \eqnok{def_N} and \eqnok{def_Tbar1}, the bound in \eqnok{bnd_comp2},
after discarding a few constant factors, is equivalent to
\beq \label{improved_compl1}
\hspace{0.1in} {\cal O} \left\{ \frac{1}{\epsilon}
\log_2 \frac{1}{\Lambda} + \frac{ \sigma^2}{\epsilon^2 } \log_2 \frac{1}{\Lambda}
+ \frac{\sigma^2 }{\epsilon} \log_2^2\frac{1}{\Lambda}
\right\}.
\eeq
Clearly, the third term of the above bound is smaller than the third term in \eqnok{improved_compl}
by a factor of $1/\Lambda$.
%
%%%%%%%%%%%%%%%%%%%%%%%%%%%%%%%%%%%%%%%%%%%%
\section{Stochastic zeroth-order methods}
\label{zeroth-order sec}
%%%%%%%%%%%%%%%%%%%%%%%%%%%%%%%%%%%%%%%%%%%%
%
In this section, we discuss how to specialize the RSPG algorithm to deal with the situations where only noisy function values
of the problem \eqnok{NLP} are available. More specifically, we assume that we can only access the noisy
zeroth-order information  of $f$ by a {\sl stochastic zeroth-order oracle} (${\cal SZO}$).
For any input $x_k$ and $\xi_k$, the ${\cal SZO}$
would output a quantity $F(x_k, \xi_k)$, where $x_k$ is the $k$-th iterate of our algorithm and $\xi_k$ is a random
variable whose distribution is supported on $\Xi \in \bbr^d$ (noting that $\Xi$ does not depend on $x_k$).
Throughout this section, we assume $F(x_k, \xi_k)$
is an unbiased estimator of $f(x_k)$, that is

{\bf A3:} For any $k \ge 1$, we have
\beqa
\bbe[F(x_k, \xi_k)] = f(x_k). \label{ass3.a}
\eeqa
We are going to apply the randomized smoothing techniques developed by \cite{DuBaMaWa11,Nest11-1}
to explore the zeroth-order information of $f$.
Hence, throughout this section, we also assume
$F(\cdot, \xi_k) \in  {\cal C}_L^{1,1}(\bbr^n)$ almost surely with respect to $\xi_k \in \Xi$, which together
with Assumption A3 imply $f\in  {\cal C}_L^{1,1}(\bbr^n)$.
Also, throughout this section, we assume that $\|\cdot\|$ is the standard Euclidean norm.

Suppose $v$ is a random vector in $\bbr^n$ with density function $\rho$, a smooth approximation of $f$ is defined as
\beq \label{rand_smooth_func-gen}
f_{\mu}(x) = \int f(x+\mu v)\rho(v) dv,
\eeq
where $\mu>0$ is the smoothing parameter. For different choices of smoothing distribution, the smoothed function
$f_{\mu}$ would have different properties. In this section, we only consider the Gaussian smoothing distribution.
That is  we assume that $v$ is a $n$-dimensional standard Gaussian random vector and
\beq \label{rand_smooth_func}
f_{\mu}(x) = \frac{1}{(2 \pi)^{\frac{n}{2}}} \int f(x+\mu v) e^{-\frac{1}{2}\|v\|^2} \,dv =\bbe_v[f(x+\mu v)].
\eeq
Nesterov \cite{Nest11-1} showed that the Gaussian smoothing approximation and $f_{\mu}$
have the following nice properties.
\begin{lemma} \label{smooth_f}
If $f\in  {\cal C}_L^{1,1}(\bbr^n)$, then
\begin{itemize}
\item [a)] $f_{\mu}$ is also Lipschitz continuously differentiable with gradient Lipschitz constant $L_{\mu} \le L$
and
\beq
\nabla f_{\mu}(x) = \frac{1}{(2 \pi)^{\frac{n}{2}}} \int  \frac{f (x + \mu v) - f(x)}{\mu} v
 e^{-\frac{1}{2}\|v\|^2} \,dv. \label{grad_fmu}
\eeq
\item [b)] for any $x \in \bbr^n$, we have
\beqa
& & \hspace{0.2in} |f_\mu(x) - f(x)| \le \frac{\mu^2}{2} L n, \label{closef} \\
& & \hspace{0.2in} \|\nabla f_{\mu}(x) - \nabla f(x)\| \le \frac{\mu}{2}L (n+3)^{\frac{3}{2}}, \label{grad_smth_close} \\
& & \hspace{0.2in} \bbe_v \left[ \left\| \frac{f (x + \mu v) - f(x)}{\mu} v \right\|^2 \right]
\le  2(n+4) \|\nabla f(x)\|^2 +
\frac{\mu^2}{2} L^2 (n+6)^3.\label{bound_grad_mu}
\eeqa
\item [c)]  $f_{\mu}$ is also convex provided $f$ is convex.
\end{itemize}
\end{lemma}

\vgap

In the following, let us define the approximated stochastic gradient of $f$ at $x_k$ as
\beq \label{grad_free_def}
G_{\mu}(x_k, \xi_k, v)=\frac{F(x_k+\mu v, \xi_k)-F(x_k,\xi_k)}{\mu}v,
\eeq
and define $G(x_k, \xi_k) = \nabla_x F(x_k, \xi_k)$. We assume the Assumption 1 holds for $G(x_k, \xi_k)$.
Then, by the Assumption A3 and Lemma~\ref{smooth_f}.a), we directly get
\beq\label{expect_Gmu}
  \bbe_{v,\xi_k} [G_{\mu}(x_k, \xi_k, v)] = \nabla f_{\mu}(x_k),
\eeq
where the expectation is taken with respect to $v$ and $\xi_k$.

Now based on the RSPG algorithm, we state an algorithm which only uses zeroth-order information
to solve problem \eqnok{NLP}.
\vskip 0.1cm

\noindent {\bf A randomized stochastic projected gradient free (RSPGF) algorithm}
\begin{itemize}
\item [] {\bf Input:} Given initial point $x_1 \in X$,
iteration limit $N$, the stepsizes $\{\gamma_k\}$ with $\gamma_k >0$, $k \ge 1$,
the batch sizes $\{m_k\}$ with $m_k > 0$, $k \ge 1$,
and the probability mass function $P_R$ supported on  $\{1,\ldots, N\}$.
\item [] {\bf Step } $0$. Let $R$ be a random variable with probability mass function $P_{R}$.
\item [] {\bf Step } $k=1, \ldots, R-1$. Call the ${\cal SZO}$ $m_k$ times to obtain $G_{\mu}(x_k, \xi_{k,i}, v_{k,i})$,
$i = 1, \ldots, m_k$, set
\beq \label{def_Gkmu}
G_{\mu,k} = \frac{1}{m_k} \sum_{i=1}^{m_k} G_{\mu}(x_k, \xi_{k,i}, v_{k,i})
\eeq
and compute
\beq \label{update_RSPGF}
x_{k+1} = \arg \min_{u \in X} \left\{\langle G_{\mu,k}, u \rangle
+\frac{1}{\gamma_k} V(u, x_k)+h(u) \right\}.
\eeq
\item [] {\bf Output:} $x_R$.
\end{itemize}

Compared with RSPG algorithm, we can see at the $k$-th iteration, the RSPGF algorithm simply replaces the stochastic
gradient $G_k$ by the approximated stochastic gradient $G_{\mu,k}$. By \eqnok{expect_Gmu},
$G_{\mu,k}$ can be simply viewed as an unbiased stochastic gradient of the smoothed function $f_{\mu}$.
However, to apply the results developed in the previous section, we still need an estimation of the bound
on the variations of the stochastic gradient $G_{\mu,k}$. In addition, the role that the
smoothing parameter $\mu$ plays and the proper selection of $\mu$ in the RSPGF algorithm are still not clear now.
We answer these questions in the following series of theorems and their corollaries.
\begin{theorem} \label{zero_main_theorem_stch}
Suppose that the stepsizes $\{\gamma_k\}$ in the RSPGF algorithm are chosen such that $ 0 < \gamma_k \le \alpha/ L$
with $\gamma_k < \alpha/L$ for at least one $k$, and the probability mass function $P_R$ are chosen as \eqnok{prob_fun}.
If $\| \nabla f(x)\| \le M$ for all $x \in X$, then under Assumptions A1 and A3,
\begin{itemize}
\item [(a)] for any $N \ge 1$, we have
\beq \label{zero_main_cnvg_stch}
\bbe[\|\bar{g}_{_{\mu, X,R}}\|^2]
\le \frac{L D_{\Psi}^2 + \mu^2 L n + (\tilde{\sigma}^2/\alpha) {\sum_{k=1}^N (\gamma_k/ m_k)}}
{{\sum_{k=1}^N (\alpha \gamma_k - L\gamma_k^2)}},
\eeq
where the expectation is taken with respect to $R$, $\xi_{[N]}$ and $v_{[N]}:= (v_1,\ldots,v_N)$,
$D_{\Psi}$ is defined in \eqnok{def_Df},
\beq \label{def_tsigma}
\tilde{\sigma}^2 = 2(n+4) [M^2+\sigma^2+\mu^2 L^2(n+4)^2],
\eeq
and
\beq \label{proj_stch_grad_free}
\bar{g}_{_{\mu,X,k}} = P_X(x_k, G_{\mu,k}, \gamma_k),
\eeq
with $P_X$ defined in\eqnok{proj_g};
\item [(b)] if, in addition, $f$ in problem \eqnok{NLP} is convex with an optimal solution $x^*$, and the stepsizes
$\{\gamma_k\}$ are non-decreasing as \eqnok{incr_stepsize},
we have
\beq \label{zero_main_cnvg_stch_cvx1}
 \hspace{0.3in} \bbe \left[\Psi(x_R) - \Psi(x^*)\right]
\le \frac{(\alpha- L \gamma_1)V(x^*,x_1)+(\tilde{\sigma}^2/2) \sum_{k=1}^N (\gamma_k^2 /m_k)}
{ \sum_{k=1}^N (\alpha \gamma_k - L\gamma_k^2)} + \mu^2 L n,
\eeq
where the expectation is taken with respect to $R$, $\xi_{[N]}$ and $v_{[N]}$.
\end{itemize}
\end{theorem}
\begin{proof}
By our assumption that $F(\cdot, \xi_k) \in  {\cal C}_L^{1,1}(\bbr^n)$ almost surely,
\eqnok{bound_grad_mu} (applying $f = F(\cdot, \xi_k)$) and Assumption 1 with  $G(x_k, \xi_k) = \nabla_x F(x_k, \xi_k)$,
we have
\beqa
\bbe_{v_k,\xi_k} [\|G_{\mu}(x_k, \xi_k, v_k)\|^2]  & = & \bbe_{\xi_k}[ \bbe_{v_k}  [\|G_{\mu}(x_k, \xi_k, v_k)\|^2]] \nn \\
&\le& 2(n+4) [\bbe_{\xi_k} [\|G(x_k, \xi)\|^2] + \frac{\mu^2}{2} L^2 (n+6)^3  \nn \\
&\le& 2(n+4) [\bbe_{\xi_k} [\|\nabla f(x_k)\|^2] +\sigma^2]+2 \mu^2 L^2 (n+4)^3. \nn
\eeqa
Then, from the above inequality, \eqnok{expect_Gmu} and $\|\nabla f(x_k) \| \le M$, we have
\beqa
& & \bbe_{v_k,\xi_k} [\|G_{\mu}(x_k, \xi_k, v_k) - \nabla f_{\mu}( x_k)\|^2]
= \bbe_{v_k,\xi_k} [\|G_{\mu}(x_k, \xi_k, v_k)\|^2] \nn \\
& \le &  2(n+4)[M^2+\sigma^2+\mu^2 L^2 (n+4)^2] = \tilde{\sigma}^2. \label{bnd_error_mu}
\eeqa
%which, in view of \eqnok{def_Gkmu} and Jensen's inequality, implis that
%\beq \label{bnd_error_mu}
%\bbe_{v_k,\xi_k} [\|G_{\mu,k} - \nabla f_{\mu}( x_k)\|^2] \le \frac{\tilde{\sigma}^2}{m_k}.
%\eeq
Now let $\Psi_{\mu}(x) = f_{\mu}(x) + h(x)$ and $\Psi_{\mu}^* = \min_{x \in X} \Psi_{\mu}(x)$. We have
from \eqnok{closef} that
\beq \label{clostPsi}
| (\Psi_{\mu}(x) - \Psi_{\mu}^*) - (\Psi(x) - \Psi^*) | \le \mu^2 L n.
\eeq
By  Lemma~\eqnok{smooth_f}.a), we have $L_\mu \le L$ and therefore $f_{\mu} \in  {\cal C}_L^{1,1}(\bbr^n)$.
With this observation, noticing \eqnok{expect_Gmu} and \eqnok{bnd_error_mu}, viewing $G_{\mu}(x_k, \xi_k, v_k)$
as a stochastic gradient of $f_{\mu}$, then by part (a) of Theorem~\ref{main_theorem_stch} we can directly get
\[
\bbe[\|\bar{g}_{_{\mu, X,R}}\|^2]
\le \frac{L D_{\Psi_{\mu}}^2 + (\tilde{\sigma}^2/\alpha) {\sum_{k=1}^N (\gamma_k/ m_k)}}{{\sum_{k=1}^N (\alpha \gamma_k - L\gamma_k^2)}},
\]
where $ D_{\Psi_{\mu}} = [(\Psi_{\mu}(x_1) - \Psi_{\mu}^*)/L]^{1/2}$ and the expectation is taken with
respect to $R$, $\xi_{[N]}$ and $v_{[N]}$. Then, the conclusion \eqnok{zero_main_cnvg_stch}
follows the above inequality and \eqnok{clostPsi}.

We now show part (b).  Since $f$ is convex, by Lemma~\eqnok{smooth_f}.c), $f_{\mu}$ is also convex.
Again by \eqnok{clostPsi}, we have
\[
\bbe \left[\Psi(x_R) - \Psi(x^*)\right] \le \bbe \left[\Psi_{\mu}(x_R) - \Psi_{\mu}(x^*)\right] + \mu^2 L n.
\]
Then, by this inequality and the convexity of $f_{\mu}$, it follows from part (b) of Theorem~\ref{main_theorem_stch}
and similar arguments in showing the part (a) of this theorem,
the conclusion \eqnok{zero_main_cnvg_stch_cvx1} holds.
\end{proof}

\vgap

Using the previous Theorem~\ref{zero_main_theorem_stch}, similar to the Corollary~\ref{RSPG_m0},
we can give the following corollary on the RSPGF algorithm with a certain constant stepsize and
batch size at each iteration.
\begin{corollary} \label{Zero_RSPG_m0}
Suppose that in the RSPGF algorithm the stepsizes $\gamma_k = \alpha/(2L)$ for all $k=1,\ldots,N$,
the batch sizes $m_k = m$ for all $k =1, \ldots, N$,
and the probability mass function $P_R$ is set to \eqnok{prob_fun}.
%If at the $k$-th iteration of the RSPGF algorithm, we take $m$ samples $\xi_{k,i}$, $ i=1,\ldots,m$ and
%use $\bar G_{\mu,m} (x_k) := \sum_{i=1}^m G_{\mu}(x_k,\xi_{k,i}, v_k)/m$ instead of $G_{\mu}(x_k,\xi_k, v_k)$
%in \eqnok{update_RSPGF},
Then under Assumptions A1 and A3, we have
\beq \label{proj_stch_gradfree_cnvg_stch}
\bbe[\|\bar g_{_{\mu,X,R}}\|^2]  \le
\frac{4L^2 D_{\Psi}^2 + 4 \mu^2 L^2 n}{\alpha^2 N} + \frac{2\tilde{\sigma}^2}{\alpha^2 m}
\eeq
and
\beq
\bbe[\|g_{_{X,R}}\|^2]  \le \frac{\mu^2 L^2 (n+3)^2}{2 \alpha^2} +
\frac{16 L^2 D_{\Psi}^2 +16 \mu^2 L^2 n}{\alpha^2 N} + \frac{12 \tilde{\sigma}^2}{\alpha^2 m}, \label{proj_gradfree_cnvg_stch}
\eeq
where the expectation is taken with respect to $R$, $\xi_{[N]}$ and $v_{[N]}$, and
$\tilde \sigma$, $\bar g_{_{\mu,X,R}}$
and $g_{_{X,R}}$ are defined in \eqnok{def_tsigma},  \eqnok{proj_stch_grad_free}
and \eqnok{proj_grad}, respectively.

If, in addition, $f$ in the problem \eqnok{NLP} is convex with an optimal solution $x^*$, then
\beq \label{proj_stch_gradfree_cnvg_stch_cvx}
\bbe \left[\Psi(x_R) - \Psi(x^*)\right]
\le \frac{2LV(x^*,x_1)}{N\alpha}+\frac{\tilde{\sigma}^2}{2Lm} + \mu^2 L n.
\eeq
\end{corollary}
\begin{proof}
%By taking $m$ samples at each iteration of the RSPGF method, Lemma~\ref{marting}.a), \eqnok{expect_Gmu} and
%\eqnok{bnd_error_mu}, we have
%\beq \label{dec_sigma2}
%\bbe_{v_k, \xi_k} \left[ \|\bar G_{\mu,m}(x_k, \xi_k, v_k) - \nabla f_{\mu}(x_k)\|^2 \right] \le \frac{\tilde{\sigma}^2}{m},
%\eeq
 \eqnok{proj_stch_gradfree_cnvg_stch} immediately follows from
\eqnok{zero_main_cnvg_stch} with $\gamma_k = \alpha/(2L)$ and $m_k =m$ for all $k=1,\ldots,N$.
Now let $g_{_{\mu,X,R}} = P_X(x_R, \nabla f_{\mu}(x_R), \gamma_R)$,
we have from \eqnok{grad_smth_close} and Proposition~\ref{lip_proj_grad} with $x=x_R$, $\gamma = \gamma_R$,
$g_1= \nabla f(x_R)$ and $g_2= \nabla f_{\mu}(x_R)$ that
\beq \label{xxx1}
\bbe[\|g_{_{X,R}}- g_{_{\mu, X, R}}\|^2] \le \frac{\mu^2 L^2(n+3)^2}{4 \alpha^2}.
\eeq
Similarly, by Proposition~\ref{lip_proj_grad} with $x=x_R$, $\gamma = \gamma_R$,
$g_1= \bar G_{\mu,k}$ and $g_2= \nabla f_{\mu}(x_R)$, we have
\beq \label{xxx2}
\bbe[\|\bar g_{_{\mu,X,R}}- g_{_{\mu, X, R}}\|^2] \le  \frac{\tilde{\sigma}^2}{ \alpha^2 m}.
\eeq
Then, it follows from \eqnok{xxx1},  \eqnok{xxx2} and  \eqnok{proj_stch_gradfree_cnvg_stch} that
\beqa
\bbe[\|g_{_{X,R}}\|^2] &\le& 2\bbe[\|g_{_{X,R}}- g_{_{\mu, X, R}}\|^2] +2 \bbe[\|g_{_{\mu, X, R}}\|^2]\nn \\
&\le& \frac{\mu^2 L^2(n+3)^2}{2 \alpha^2} + 4\bbe[\|g_{_{\mu, X, R}}-\bar g_{_{\mu, X, R}}\|^2]+ 4\bbe[\|\bar g_{_{\mu, X, R}}\|^2] \nn \\
&\le& \frac{\mu^2 L^2(n+3)^2}{2 \alpha^2} + \frac{12 \tilde{\sigma}^2}{\alpha^2 m}
+\frac{16 L^2 D_{\Psi}^2+16 \mu^2 L^2 n}{\alpha^2 N}. \nn
\eeqa

Moreover, if $f$ is convex, then \eqnok{proj_stch_gradfree_cnvg_stch_cvx} immediately follows from \eqnok{zero_main_cnvg_stch_cvx1},
and the constant stepsizes $\gamma_k = \alpha/(2L)$ for all $k=1,\ldots,N$.
\end{proof}

\vgap

Similar to the Corollary~\ref{RSPG_m0} for the RSPG algorithm, the above results also depend on the number of
samples $m$ at each iteration. In addition, the above results depend on the smoothing parameter $\mu$ as well.
The following corollary, analogous to the Corollary~\ref{RSPG_m}, shows how to choose $m$ and $\mu$ appropriately.

\begin{corollary} \label{Zero_RSPG_m}
Suppose that all the conditions in Corollary~\ref{Zero_RSPG_m0} are satisfied.
Given a fixed total number of calls to the ${\cal SZO}$ $\bar N$, if
the smoothing parameter satisfies
\beq \label{def_mu}
\mu \le \frac{D_{\Psi}}{\sqrt{(n+4) \bar N}},
\eeq
and the number of calls to the ${\cal SZO}$  at each iteration of the RSPGF method is
\beq \label{def_m2}
m =\left \lceil \min \left\{ \max \left\{\frac{\sqrt{(n+4)(M^2 + \sigma^2) \bar N}}{  L \tilde D}, n+4 \right\},
\bar N \right\} \right \rceil,
\eeq
for some $\tilde D >0$, then we have $ (\alpha^2/L) \; \bbe[\|g_{_{X,R}}\|^2] \le \bar{\cal B}_{\bar N}$, where
\beq \label{proj_stch_gradfree_cnvg_stch_m}
\bar{\cal B}_{\bar N}:=\frac{(24\theta_2+41) L D_{\Psi}^2 (n+4) }{\bar N}+
\frac{32 \sqrt{(n+4)(M^2 + \sigma^2)}}{\sqrt{\bar N}} \left( \frac{D_{\Psi}^2}{\tilde D} +
\tilde D \theta_1   \right),
\eeq
and
\beq \label{theta_factor}
\theta_1 = \max \left\{1, \frac{\sqrt{(n+4)(M^2 + \sigma^2)}}{ L \tilde D \sqrt{\bar N}} \right\} \quad
\mbox{and} \quad \theta_2=\max \left\{1, \frac{n+4}{\bar N} \right\}.
\eeq

If, in addition, $f$ in the problem \eqnok{NLP} is convex and the smoothing parameter satisfies
\beq \label{def_mu2}
\mu \le \sqrt{\frac{V(x^*,x_1)}{\alpha (n+4) \bar N}},
\eeq
then $\bbe[\Psi(x_R)-\Psi(x^*)] \le \bar{\cal C}_{\bar N}$, where $x^*$ is an optimal solution and
\beq \label{proj_stch_gradfree_cnvg_stch_cvx_m} \hspace{0.2in}
\bar{\cal C}_{\bar N}:=\frac{(5 + \theta_2) L V(x^*,x_1)(n+4)}{\alpha \bar N} +
\frac{\sqrt{(n+4)(M^2 + \sigma^2)}}{\alpha \sqrt{\bar N}} \left( \frac{4 V(x^*,x_1)}{\tilde D} +
\alpha \tilde D \theta_1  \right).
\eeq
\end{corollary}
\begin{proof}
By the definitions of $\theta1$ and $\theta_2$ in \eqnok{theta_factor} and $m$ in \eqnok{def_m2}, we have
\beq \label{def_m2New}
 m = \left \lceil  \max \left\{\frac{\sqrt{(n+4)(M^2 + \sigma^2) \bar N}}{  L \tilde D \theta_1},
\frac{n+4}{\theta_2} \right\}  \right \rceil.
\eeq
Given the total number of calls to the ${\cal SZO}$ $\bar{N}$ and the the number
$m$ of calls to the ${\cal SZO}$ at each iteration, the RSPGF algorithm
can perform at most $N = \lfloor \bar{N}/m \rfloor$ iterations. Obviously, $N  \ge \bar{N}/(2m)$.
With this observation $\bar{N} \ge m$, $\theta_1 \ge 1$ and $\theta_2 \ge 1$,
by \eqnok{proj_gradfree_cnvg_stch}, \eqnok{def_mu} and \eqnok{def_m2New},
we have
\beqa
& & \bbe[\|g_{_{X,R}}\|^2]  \nn \\
&\le&  \frac{L^2 D_{\Psi}^2 (n+3)}{2 \alpha^2 \bar N} +
\frac{24(n+4)(M^2+\sigma^2)}{\alpha^2 m}+ \frac{24 L^2 D_{\Psi}^2 (n+4)^2}{\alpha^2 m \bar N}
 +\frac{32 L^2 D_{\Psi}^2 m}{\alpha^2 \bar {N}} \left(1+\frac{1}{\bar N}\right) \nn \\
&\le& \frac{L^2 D_{\Psi}^2 (n+4)}{2 \alpha^2 \bar N} + \frac{24 \theta_1 L \tilde D \sqrt{(n+4)(M^2+\sigma^2)}}
{\alpha^2 \sqrt{\bar N}}+\frac{24 \theta_2 L^2 D_{\Psi}^2 (n+4)}{\alpha^2 \bar N} \nn \\
&& + \frac{32 L^2 D_{\Psi}^2}{\alpha^2 \bar{N}} \left( \frac{\sqrt{(n+4)(M^2 + \sigma^2) \bar N}}{  L \tilde D \theta_1}
+ \frac{n+4}{\theta_2} \right) + \frac{32 L^2 D_{\Psi}^2}{\alpha^2 \bar{N}} \nn \\
& \le & \frac{L^2 D_{\Psi}^2 (n+4)}{2 \alpha^2 \bar N} + \frac{24 \theta_1 L \tilde D \sqrt{(n+4)(M^2+\sigma^2)}}
{\alpha^2 \sqrt{\bar N}}+\frac{24 \theta_2 L^2 D_{\Psi}^2 (n+4)}{\alpha^2 \bar N} \nn \\
&& + \frac{32  L D_{\Psi}^2 \sqrt{(n+4)(M^2 + \sigma^2)}}{ \alpha^2 \tilde{D} \sqrt{\bar{N}} }
+  \frac{32  L^2 D_{\Psi}^2 (n+4)}{  \alpha^2 \bar{N}} + \frac{32 L^2 D_{\Psi}^2}{\alpha^2 \bar{N}},
\nn
\eeqa
which after integrating the terms give \eqnok{proj_stch_gradfree_cnvg_stch_m}.
The conclusion \eqnok{proj_stch_gradfree_cnvg_stch_cvx_m} follows similarly by \eqnok{def_mu2} and
\eqnok{proj_stch_gradfree_cnvg_stch_cvx}.
\end{proof}

\vgap

We now would like to add a few remarks about the above the results in Corollary~\ref{Zero_RSPG_m}.
Firstly, the above complexity bounds are similar to those of the first-order RSPG method in Corollary~\ref{RSPG_m}
in terms of their dependence on the total number of stochastic oracle $\bar N$ called by the algorithm.
However, for the zeroth-order case, the complexity in Corollary~\ref{Zero_RSPG_m} also depends on the size of the
gradient $M$ and the problem dimension $n$. Secondly, the value of $\tilde D$ has not been specified.
It can be easily seen from \eqnok{proj_stch_gradfree_cnvg_stch_m} and \eqnok{proj_stch_gradfree_cnvg_stch_cvx_m} that
when $\bar{N}$ is relatively large such that
$\theta_1 = 1$ and $\theta_2 =1$, i.e.,
\beq\label{zero_large-barN}
\bar N \ge \max \left\{ \frac{(n+4)^2(M^2 + \sigma^2)}{L^2 \tilde{D}^2}, n+4 \right\},
\eeq
the optimal choice of $\tilde D$ would be $D_{\Psi}$ and $2\sqrt{V(x^*,x_1)/\alpha}$ for solving nonconvex and
convex SP problems, respectively. With this selection of $\tilde D$, the bounds in \eqnok{proj_stch_gradfree_cnvg_stch_m}
and \eqnok{proj_stch_gradfree_cnvg_stch_cvx_m}, respectively, reduce to
\beq \label{nocvx_stch_m2}
\frac{\alpha^2}{L} \bbe[\|g_{_{X,R}}\|^2]
\le \frac{ 65 L D_{\Psi}^2 (n+4) }{\bar N}+
\frac{64 \sqrt{(n+4)(M^2 + \sigma^2)}}{\sqrt{\bar N}}
\eeq
and
\beq \label{cvx_stch_m2} \hspace{0.2in}
\bbe[\Psi(x_R)-\Psi(x^*)] \le \frac{6 L V(x^*,x_1)(n+4)}{\alpha \bar N} +
\frac{4 \sqrt{V(x^*,x_1)(n+4)(M^2 + \sigma^2)}}{\sqrt{ \alpha \bar N}}.
\eeq
Thirdly, the complexity result in \eqnok{proj_stch_gradfree_cnvg_stch_cvx_m}
implies that when $f$ is convex, if $\epsilon$ sufficiently small, then the number of calls to the
${\cal SZO}$ to find a solution $\bar x$ such that $\bbe[f(\bar x) - f^*] \le \epsilon$ can be bounded
by ${\cal O}(n/\epsilon^2)$, which is better than the complexity of ${\cal O}(n^2/\epsilon^2)$ established
by Nesterov \cite{Nest11-1} to find such a solution for general convex SP problems.
%
%%%%%%%%%%%%%%%%%%%%%%%%%%%%%%%%%%%%%%%%%%%%
\section{Numerical Results}
\label{sec_num_result}
%%%%%%%%%%%%%%%%%%%%%%%%%%%%%%%%%%%%%%%%%%%%
In this section, we present the numerical results of our computational experiments
for solving two SP problems: a stochastic least square problem with a nonconvex regularization term and
a stochastic nonconvex semi-supervised support vector machine problem.

\emph{Algorithmic schemes.}
We implement the RSPG algorithm and its two-phase variant $2$-RSPG algorithm described in Section~\ref{sec_stch_first},
where the prox-function $V(x,z)=\|x-z\|^2/2$, the stepsizes $\gamma_k = \alpha/(2L)$ with $\alpha =1$
for all $k \ge 1$, and the probability mass function $P_R$ is set to \eqnok{prob_fun}.
Also, in the optimization phase of the $2$-RSPG algorithm, we take $S=5$ independent runs of the RSPG algorithm to compute $5$
candidate solutions. Then, we use an i.i.d. sample of size $T=N/2$ in the post-optimization phase to estimate the
projected gradients at these candidate solutions and then choose the best one, $\bar{x}^*$, according to \eqnok{post_opt}.
Finally, the solution quality at $\bar{x}^*$ is evaluated by using another i.i.d. sample of size $K >> N$.
In addition to the above algorithms, we also implement another variant of the $2$-RSPG algorithm, namely,
$2$-RSPG-V algorithm. This algorithm also consists of two phases similar to the $2$-RSPG algorithm.
In the optimization phase, instead of terminating the RSPG algorithm by using a random count $R$, we terminate the
algorithm by using a fixed number of iterations, say $NS$. We then randomly pick up $S=5$ solutions from the
generated trajectory according to $P_R$ defined in \eqnok{prob_fun}.
The post-optimization phase of the $2$-RSPG-V algorithm is the same as that of the $2$-RSPG algorithm.
Note that, in the $2$-RSPG-V algorithm, unlike the $2$-RSPG algorithm, the $S$ candidate solutions are not independent
and hence, we cannot provide the large-deviation results similar to the $2$-RSPG algorithm. We also implement the
RSG, $2$-RSG and $2$-RSG-V algorithms developed in \cite{GhaLan12} to compare with our results.

\vgap
\noindent \emph{Estimation of parameters.}
We use an initial i.i.d. sample of size $N_0=200$ to estimate the problem parameters, namely, $L$ and $\sigma$.
We also estimate the parameter $\tilde D = D_{\Psi}$ by \eqnok{def_Df}. More specifically, since
the problems considered in this section have nonnegative optimal values, i.e.,
$\Psi^* \ge 0$, we have $D_{\Psi} \le (2 \Psi(x_1) / L)^\frac{1}{2}$, where $x_1$ denotes the starting point of
the algorithms.

\vgap

\noindent \emph{Notation in the tables.}
\begin{itemize}
\item $N S$ denotes the maximum number of calls to the stochastic oracle performed in the optimization phase  of
the above algorithms. For example, $N S = 1,000$ has the following implications.
\begin{itemize}
\item For the RSPG algorithm, the number of samples per iteration $m$ is computed according to \eqnok{def_m} with
$\bar N=1000$ and the iteration limit $N$ is set to $\lfloor 1000/m \rfloor$;
\item For the $2$-RSPG algorithm, since $S=5$, we set $\bar N = 200$. The $m$ and $N$ are computed as mentioned above.
In this case, total number of calls to the stochastic oracle will be at most $1,000$ (this does
not include the samples used in the post optimization phase);
\item For the $2$-RSPG-V algorithm, $m$ is computed  as mentioned above  and we run the RSPG method for
$\lfloor 1000/m \rfloor$ iterations and randomly select $S = 5$ solutions from the trajectory according to
$P_R$ defined in \eqnok{prob_fun}.
\end{itemize}
\item $\bar{x}^*$ is the output solution of the above algorithms.
\item \emph{Mean} and  \emph{Var.} represent, respectively, the average and variance of the results obtained over
 $20$ runs of each algorithm.
\end{itemize}

\begin{table}
\caption{Estimated $\|\nabla f(\bar{x}^*)\|^2$ for the least square problem ($K = 75,000$)}

\centering
\label{TB1}
\footnotesize
\begin{tabular}{|c|c|c|c|c||c|c|c|}
\hline
$\bar N S$&&RSG&2-RSG&2-RSG-V&RSPG&2-RSPG&2-RSPG-V \\[2ex]
\hline
&&\multicolumn{6}{|c|}{$n=100, \tilde {\sigma}=0.1$}\\
\hline
\multirow{2}{*}{1000}&mean&0.2509&0.3184&0.0794&0.1564&0.3176&0.0422\\
&var.&4.31e-2&1.68e-2&1.23e-3&4.58e-2&2.54e-2&8.99e-3 \\
\hline
\multirow{2}{*}{5000}&mean&0.0828&0.0841&0.0042&0.0113&0.0164&0.0009\\
&var.&6.75e-3&1.03e-3 &1.35e-5&4.22e-4&3.37e-4&4.36e-8\\
\hline
\multirow{2}{*}{25000}&mean&0.0056&0.0070&0.0002&0.0006&0.0010&0.0004 \\
&var.&1.69e-4&1.08e-4&3.41e-8&2.05e-7&1.43e-7&7.83e-9 \\
\hline
\hline
&&\multicolumn{6}{|c|}{$n=100, \tilde {\sigma}=1$}\\
\hline
\multirow{2}{*}{1000}&mean&0.3731&0.3761&0.1230&0.2379&0.3567&0.0364 \\
&var.&3.38e-2&1.40e-2&3.28e-3&4.01e-2&1.41e-2&1.24e-3\\
\hline
\multirow{2}{*}{5000}&mean&0.1095&0.1314&0.0135&0.0436&0.0323&0.0075 \\
&var.&2.22e-2&3.96e-3&4.67e-5&1.44e-2&8.69e-4&7.97e-5 \\
\hline
\multirow{2}{*}{25000}&mean&0.0374&0.0172&0.0078&0.0138&0.0048&0.0046\\
&var.&8.46e-3&1.83e-4&4.54e-4&1.95e-3&8.48e-7&5.60e-5 \\
\hline
&&\multicolumn{6}{|c|}{$n=500, \tilde {\sigma}=0.1$}\\
\hline
\multirow{2}{*}{1000}&mean&0.5479&0.6865&0.4121&0.4212&0.8977&0.2579\\
&var.&3.47e-2&6.17e-3&1.09e-2&5.13e-2&2.64e-3&1.34e-2\\
\hline
\multirow{2}{*}{5000}&mean&0.2481&0.3560&0.0873&0.1030&0.1997&0.0154 \\
&var.&4.38e-2&3.45e-3&1.28e-3&2.57e-2&2.21e-3&1.83e-4 \\
\hline
\multirow{2}{*}{25000}&mean&0.2153&0.0876&0.0084&0.1093&0.0136&0.0011 \\
&var.&6.77e-2&1.13e-3&3.97e-5&4.07e-2&3.24e-5&3.04e-8\\
\hline
&&\multicolumn{6}{|c|}{$n=500, \tilde {\sigma}=1$}\\
\hline
\multirow{2}{*}{1000}&mean&0.5869&0.7444&0.4828&0.4371&0.7771&0.4190\\
&var.&2.14e-2&4.18e-3&9.40e-3&3.40e-2&5.15e-3&4.13e-2\\
\hline
\multirow{2}{*}{5000}&mean&0.3603&0.4732&0.1699&0.1745&0.2987&0.0411 \\
&var.&3.77e-2&8.13e-3&1.22e-3&3.51e-2&1.87e-2&6.21e-4 \\
\hline
\multirow{2}{*}{25000}&mean&0.2467&0.1584&0.0342&0.1271&0.0351&0.0189 \\
&var.&6.49e-2&1.87e-3&3.72e-4&4.30e-2&2.83e-4&3.89e-5\\
\hline
&&\multicolumn{6}{|c|}{$n=1000, \tilde {\sigma}=0.1$}\\
\hline
\multirow{2}{*}{1000}&mean&1.853&2.417&1.549&1.855&3.092&1.937\\
&var.&1.73e-1&1.31e-2&1.62e-2&1.88e-1&1.29e-1&2.64e-1\\
\hline
\multirow{2}{*}{5000}&mean&0.9555&1.501&0.5422&0.4944&1.832&0.1368 \\
&var.&3.62e-1&6.39e-2&3.73e-2&4.82e-1&2.36e-1&8.78e-3 \\
\hline
\multirow{2}{*}{25000}&mean&0.6305&0.4725&0.0839&0.3402&0.1100&0.0071 \\
&var.&6.38e-1&2.08e-2&1.19e-2&4.40e-1&4.54e-3&1.97e-4\\
\hline
&&\multicolumn{6}{|c|}{$n=1000, \tilde {\sigma}=1$}\\
\hline
\multirow{2}{*}{1000}&mean&1.868&2.407&1.560&1.701&3.208&1.662\\
&var.&1.44e-1&1.22e-2&4.37e-2&1.84e-1&1.54e-1&2.75e-1\\
\hline
\multirow{2}{*}{5000}&mean&1.297&1.596&0.6438&0.8032&1.403&0.2408 \\
&var.&5.25e-1&5.26e-2&3.04e-2&6.38e-1&1.10e-1&3.26e-2 \\
\hline
\multirow{2}{*}{25000}&mean&0.575&0.6309&0.0793&0.2079&0.1806&0.0336 \\
&var.&3.43e-1&4.65e-2&1.38e-3&1.17e-1&1.43e-2&3.67e-6\\
\hline
\end{tabular}
\end{table}

\begin{table}
\caption{Average ratio of true recovered zeros for the penalized least square problem}

\centering
\label{TB2}
\footnotesize
\begin{tabular}{|c|c|c|c||c|c|c|}
\hline
$\bar{N}S$&RSG&2-RSG&2-RSG-V&RSPG&2-RSPG&2-RSPG-V \\[2ex]
\hline
&\multicolumn{6}{|c|}{$n=100, \tilde {\sigma}=0.1$}\\
\hline
1000&0.18&0.14&0.11&0.17&0.13&0.19\\
\hline
5000&0.26&0.11 &0.60&0.56&0.19&0.98\\
\hline
25000&0.82&0.56&0.98&0.97&0.95&1.00 \\
\hline
\hline
&\multicolumn{6}{|c|}{$n=100, \tilde {\sigma}=1$}\\
\hline
1000&0.16&0.14&0.09&0.12&0.11&0.09 \\
\hline
5000&0.16&0.09&0.21&0.14&0.08&0.15 \\
\hline
25000&0.41&0.20&0.51&0.26&0.16&0.24 \\
\hline
&\multicolumn{6}{|c|}{$n=500, \tilde {\sigma}=0.1$}\\
\hline
1000&0.14&0.17&0.09&0.08&0.27&0.06\\
\hline
5000&0.18&0.08&0.22&0.31&0.06&0.56 \\
\hline
25000&0.47&0.21&0.84&0.63&0.55&0.99\\
\hline
&\multicolumn{6}{|c|}{$n=500, \tilde {\sigma}=1$}\\
\hline
1000&0.16&0.23&0.10&0.09&0.17&0.12\\
\hline
5000&0.16&0.10&0.17&0.12&0.06&0.16 \\
\hline
25000&0.39&0.17&0.60&0.33&0.17&0.45\\
\hline
&\multicolumn{6}{|c|}{$n=1000, \tilde {\sigma}=0.1$}\\
\hline
1000&0.1&0.17&0.07&0.05&0.25&0.10\\
\hline
5000&0.10&0.05&0.09&0.10&0.04&0.11 \\
\hline
25000&0.31&0.09&0.51&0.55&0.12&0.91\\
\hline
&\multicolumn{6}{|c|}{$n=1000, \tilde {\sigma}=1$}\\
\hline
1000&0.12&0.14&0.05&0.09&0.20&0.08\\
\hline
5000&0.11&0.06&0.09&0.08&0.03&0.08 \\
\hline
25000&0.20&0.09&0.40&0.28&0.09&0.45\\
\hline
\end{tabular}
\end{table}

\subsection {Nonconvex least square problem}
In our first experiment, we consider the following least square problem with a smoothly clipped absolute deviation
penalty term given in \cite{FanLi01-1}:
\[
\min_{x \in \bbr^n} f(x) := \bbe_{u,v}[(\langle x, u \rangle - v)^2] + \sum_{j=1}^d p_\lambda (|x_j|) .
\]
Here, the penalty term $p_{\lambda}: \bbr_+ \to \bbr$ satisfies $p_\lambda (0)=0$ and has derivatives as
\[
p'_\lambda (\beta)= \lambda \left\{I(\beta \le \lambda)+\frac{\max(0, a \lambda-\beta)}
{(a-1)\lambda}I(\beta > \lambda) \right\},
\]
where $a>2$ and $\lambda>0$ are constant parameter, and $I$ is the indicator function.
As it can be seen, $p_\lambda (|\cdot|)$ is nonconvex and non-differentiable at $0$.
Therefore, we replace $p_{\lambda}$ by its smooth nonconvex approximation $q_\lambda: \bbr_+ \to \bbr$,
satisfying  $q_\lambda (0)=0$ with derivative defined as
\[
q'_\lambda (\beta)= \left\{\beta I(\beta \le \lambda)+\frac{\max(0, a \lambda-\beta)}{(a-1)}I(\beta > \lambda) \right\}.
\]
In this experiment, we assume that $u$ is a sparse vector, whose components are standard normal, and $v$ is obtained by
$v = \langle \bar x, u \rangle +\xi$, where $\xi \sim N(0,\bar{\sigma}^2)$ is the random noise independent of
$u$ and the coefficient $\bar x$ defines the true linear relationship between $u$ and $v$. Also,
we set the parameters to $a=3.7$ and $\lambda=0.01$ in the numerical experiments.

We consider three different problem sizes with $n=100, 500$ and $1,000$, and two different noise levels with
$\bar{\sigma}=0.1$ and  $1$.
Also, we set the data sparsity to $5 \%$, which means that approximately five percent of $u$ is nonzero for each data point
$(u, v)$. We also use a sparse multivariate standard normal $\bar x$ for generating data points. Also, for all problem sizes,
the initial point is set to $x_1 = 5* \bar{x}_0 \in \bbr^n$, where $\bar{x}_0$ is a multivariate standard normal vector
with approximately $10 \%$ nonzero elements. Since this problem is unconstrained, we also implement variants of the RSG
algorithm developed in \cite{GhaLan12}.
Table~\ref{TB1} shows the mean and variance of the 2-norm of the gradient at the solutions returned by $20$ runs of
the comparing algorithms. Moreover, we are also interested in recovering sparse solutions. Hence,
we set the component of output solutions to be zero if its absolute value is less than a threshold $0.02$. We call such solutions as the recovered truncated zeros and compute their ratio with respect to the number of true zeros.
Table~\ref{TB2} shows the average of this ratio over $20$ runs of the algorithms.
The following observations can be made from the numerical results.
\begin{itemize}
\item {\bf Different variants of the RSPG algorithm:} Firstly, over  $20$ runs of the algorithm, the solutions of
the RSPG algorithm has relatively large variance. Secondly, both $2$-RSPG and $2$-RSPG-V can significantly reduce the
variance of the RSPG algorithm.
Moreover, for a given fixed $NS$, the solution quality of the $2$-RSPG-V algorithm is significantly better
than that of the $2$-RSPG algorithm.
The reason might be that, for fixed $N S$, $S=5$ times more iterations are being used in the $2$-RSPG-V algorithm
to generate new solutions in the trajectory.
\item {\bf Different variants of the RSG algorithm:} The differences among different variants of the RSG algorithm
are similar to those of the corresponding variants of the RSPG algorithm.
\item {\bf RSPG algorithm vs. RSG algorithm:} In terms of the mean value, the solutions give by RSG and RSPG algorithms
are comparable. However, the solution of the RSPG algorithm usually have less variance than that of the corresponding
RSG algorithm. The possible reason is that we use a better approximation for stochastic gradient by incorporating
mini-batch of samples during the execution of the RSPG method.
\end{itemize}

\subsection {Semi-supervised support vector machine problem}
In this second experiment, we consider a binary classification problem. The training set is divided to two types of
data, which consists of labeled and unlabeled examples, respectively. The linear semi-supervised support vector machine
problem can be formulated as follows \cite{ChaSinKee08-1}:
\beqa
\min_{b \in \bbr, \; x \in \bbr^n} f(x,b) &:= & \lambda_1 \bbe_{u_1,v}
\left[\max\left\{0,1-v(\langle x, u_1 \rangle+b)\right\}^2\right] \nn \\
 & & + \lambda_2 \bbe_{u_2} \left[\max\left\{0,1-|\langle x, u_2 \rangle+b|\right\}^2\right]+\lambda_3\|x\|_2^2, \nn
\eeqa
where $(u_1,v)$ and $u_2$ are labeled and unlabeled examples, respectively.
Clearly, the above problem is nonsmooth, nonconvex, and does not fit the setting of the problem \eqnok{NLP}.
Using a smooth approximation of the above problem \cite{ChaSinKee08-1}, we can reformulate it as
\beqa \label{S3VM_def}
\min_{(x,b) \in \bbr^{n+1}}  f(x,b) &:=& \bbe_{u_1,u_2,v}\left[\lambda_1 \max\left\{0,1-v(\langle x, u_1
\rangle+b)\right\}^2 + \lambda_2 e^{-5 \left\{\langle x, u_2 \rangle+b\right\}^2}\right] \nn \\
& & +\lambda_3\|x\|_2^2.
\eeqa
Here, we assume that the feature vectors $u_1$ and $u_2$ are drawn from standard normal distribution with approximately
$5 \%$ nonzero elements. Moreover, we assume that label $v \in \{0,1\}$ with $v= \mbox{sgn}(\langle \bar x, u' \rangle+b)$
for some $\bar x \in \bbr^n$. The parameters are also set to $\lambda_1=1, \ \ \lambda_2=0.5$ and $\lambda_3=0.5$.
The choices of problem size and the noise variance are same as those of the nonconvex
penalized least square problem in the previous subsection.
We also want to determine the labels of unlabeled examples such that the ratio of new positive labels is close to
that of the already labeled examples. It is shown in \cite{ChaSinKee08-1} that if the examples
come from a distribution with zero mean, then, to have balanced new labels, we can consider the following constraint
\beq \label{bnd_b}
|b-2r+1|  \le \delta,
\eeq
where $r$ is the ratio of positive labels in the already labeled examples and $\delta$ is a tolerance
setting to $0.1$ in our experiment.
Therefore, \eqnok{S3VM_def} together with the constraint \eqnok{bnd_b} is a constrained nonconvex smooth problem,
which fits the setting of problem \eqnok{NLP}.
Table~\ref{TB3} shows the mean and variance of the 2-norm of the projected gradient at the solutions obtained by
20 runs of the RSPG algorithms, and Table~\ref{TB4} gives the corresponding average objective values at the
solutions given in Table~\ref{TB3}.
 Similar to the conclusions in the previous subsection, we again can see $2$-RSPG-V algorithm has
the best performance among the variants of the RSPG algorithms and $2$-RSPG algorithms is more stable than the RSPG algorithm.

\begin{table}
\caption{Estimated $\|g_{_X} (\bar{x}^*)\|^2 $ for the semi-supervised support vector machine
problem ($K=75,000$)}
\vgap

\centering
\label{TB3}
\footnotesize
\begin{tabular}{|c|c|c|c|c|}
\hline
$\bar {N} S$&&RSPG&2-RSPG&2-RSPG-V\\
\hline
\multicolumn{5}{|c|}{$n=100$}\\
\hline
\multirow{2}{*}{1000}&mean&1.355&0.2107&0.1277 \\
&var.&1.21e+1&9.50e-3&5.45e-3 \\
\hline
\multirow{2}{*}{5000}&mean&0.1032&0.1174&0.0899 \\
&var.&4.96e-2&4.42e-3&6.28e-3 \\
\hline
\multirow{2}{*}{25000}&mean&0.0352&0.0699&0.0239 \\
&var.&1.13e-3&3.42e-3&1.73e-5 \\
\hline
\hline
\multicolumn{5}{|c|}{$n=500$}\\
\hline
\multirow{2}{*}{1000}&mean&5.976&0.7955&0.1621 \\
&var.&1.93e+2&6.07e-1&1.15e-3 \\
\hline
\multirow{2}{*}{5000}&mean&0.2237&0.1703&0.0928 \\
&var.&2.77e-1&4.39e-3&1.29e-3 \\
\hline
\multirow{2}{*}{25000}&mean&0.2174&0.0832&0.0339 \\
&var.&2.35e-1&2.41e-4&8.04e-6 \\
\hline
\hline
\multicolumn{5}{|c|}{$n=1000$}\\
\hline
\multirow{2}{*}{1000}&mean&27.06&2.417&0.3167 \\
&var.&6.00e+3&1.73e+1&1.19e-2 \\
\hline
\multirow{2}{*}{5000}&mean&16.24&0.4726&0.1463 \\
&var.&2.20e+3&2.85e+1&1.46e-3 \\
\hline
\multirow{2}{*}{25000}&mean&0.1007&0.1378&0.0672 \\
&var.&2.46e-2&5.63e-5&5.10e-5 \\
\hline
\end{tabular}
\end{table}

\begin{table}
\caption{Average objective values at $\bar{x}^*$, obtained in Table~\ref{TB3}}
\centering
\label{TB4}
\footnotesize
\begin{tabular}{|c|c|c|c|}
\hline
$\bar {N} S$&RSPG&2-RSPG&2-RSPG-V\\
\hline
\multicolumn{4}{|c|}{$n=100$}\\
\hline
1000&1.497&0.9331&0.9094 \\
\hline
5000&0.9131&0.9078&0.8873 \\
\hline
25000&0.8736&0.8862&0.8728 \\
\hline
\hline
\multicolumn{4}{|c|}{$n=500$}\\
\hline
1000&3.813&1.364&1.038 \\
\hline
5000&1.062&1.043&0.9998 \\
\hline
25000&1.062&0.9982&0.9719 \\
\hline
\hline
\multicolumn{4}{|c|}{$n=1000$}\\
\hline
1000&14.05&2.100&1.055\\
\hline
5000&8.77&1.128&0.9767\\
\hline
25000&0.9513&0.9719&0.9351\\
\hline
\end{tabular}
\end{table}

% ------- Conclusion -----------------------------------------------------------
\section{Conclusion}\label{conclusion}
This paper proposes a new stochastic approximation algorithm with its variants for
solving a class of nonconvex stochastic composite optimization problems.
This new randomized stochastic projected gradient (RSPG) algorithm uses mini-batch of samples
at each iteration to handle the constraints. The proposed algorithm is set up in a way that
a more general gradient projection according to the geometry of the constraint set
could be used. The complexity bound of our algorithm is established in a unified way, including both convex and nonconvex
objective functions. Our results show that the RSPG algorithm would automatically
maintains a nearly optimal rate of convergence for solving stochastic convex programming problems.
To reduce the variance of the RSPG algorithm, a two-phase RSPG algorithm  is also
proposed. It is shown that with a special post-optimization phase, the variance of the
the solutions returned by the RSPG algorithm could be significantly reduced, especially
when a light tail condition holds. Based on this RSPG algorithm, a stochastic projected gradient
free algorithm, which only uses the stochastic zeroth-order information, has been also proposed
and analyzed.  Our preliminary numerical results show that our two-phase RSPG algorithms, the 2-RSPG
and its variant 2-RSPG-V algorithms, could be very effective and stable for solving the aforementioned
nonconvex stochastic composite optimization problems.
% ------- Conclusion ---------------------------------------------------------
\bibliographystyle{siam}
\bibliography{../glan-bib}
\end{document}